\documentclass[oneside,11pt]{amsart}
\usepackage{amsmath}
\usepackage{amssymb}
\usepackage{graphicx}
\usepackage{caption}
\usepackage{subcaption}

\newtheorem{thm}{Theorem}[section]
\newtheorem{prop}[thm]{Proposition}
\newtheorem{lem}[thm]{Lemma}
\newtheorem{cor}[thm]{Corollary}

\theoremstyle{definition}
\newtheorem{defn}[thm]{Definition}
\newtheorem{rem}[thm]{Remark}

\newtheorem{exmp}[thm]{Example}
\newtheorem{ques}[thm]{Question}

\newcommand{\abs}[1]{\lvert{#1}\rvert}

\renewcommand{\bar}[1]{\overline{#1}}

\newcommand{\bigset}[2]{ \bigl\{ \, {#1} \bigm| {#2} \, \bigr\} }
\renewcommand{\emptyset}{\varnothing}

\newcommand{\field}[1]{\mathbb{#1}}

\newcommand{\PP}{\field{P}}

\newcommand{\EE}{\field{E}}
\newcommand{\NN}{\field{N}}
\newcommand{\RR}{\field{R}}

\DeclareMathOperator{\CAT}{CAT}

\newcommand{\ball}[2]{B ( {#1}, {#2} )}%Ball{center}{radius}

\newcommand{\nbd}[2]{\mathcal{N}_{#2}({#1})} % Neighborhood{center}{radius}

\newcommand{\Set}[1]{\mathcal{#1}}

\DeclareMathOperator{\Cayley}{Cayley}
\DeclareMathOperator{\diam}{diam}
\DeclareMathOperator{\Hull}{Hull}

\begin{document}

\title[Relative divergence of finitely generated groups]{Relative divergence of finitely generated groups}

\author{Hung Cong Tran}
\address{Dept.\ of Mathematical Sciences\\
University of Wisconsin--Milwaukee\\
P.O.~Box 413\\
Milwaukee, WI 53201\\
USA}
\email{hctran@uwm.edu}

\date{\today}

\begin{abstract}
We generalize the concept of divergence of finitely generated groups by introducing the upper and lower relative divergence of a finitely generated group with respect to a subgroup. Upper relative divergence generalizes Gersten's notion of divergence, and lower relative divergence generalizes a definition of Cooper-Mihalik. While the lower divergence of Cooper-Mihalik can only be linear or exponential, relative lower divergence can be any polynomial or exponential function. In this paper, we examine the relative divergence (both upper and lower) of a group with respect to a normal subgroup or a cyclic subgroup. We also explore relative divergence of $\CAT(0)$ groups and relatively hyperbolic groups with respect to various subgroups to better understand geometric properties of these groups.
\end{abstract}

\subjclass[2000]{%
20F67, % Hyperbolic groups and nonpositively curved groups
20F65} % Geometric group theory
\maketitle

\section{Introduction}
Two different notions of divergence of a finitely generated group are introduced by Cooper-Mihalik \cite{MR1170363} and Gersten \cite{MR1254309}. We refer to Cooper-Mihalik's notion as lower divergence and Gersten's notion as upper divergence. The lower divergence of a one-ended group $G$ is exponential if $G$ is hyperbolic and linear otherwise (see Cooper-Mihalik \cite{MR1170363} and Sisto \cite{Sisto}). Therefore, lower divergence only detects hyperbolicity. Upper divergence is more diverse since the upper divergence of a finitely generated group can be any polynomial or exponential function (see Macura \cite{MR3032700} and Sisto \cite {Sisto}). Upper divergence has been studied by Macura \cite{MR3032700}, Behrstock-Charney \cite{MR2874959}, Duchin-Rafi \cite {MR2563768}, Dru{\c{t}}u-Mozes-Sapir \cite{MR2584607}, Sisto \cite {Sisto} and others. Moreover, upper divergence is a quasi-isometry invariant, and it is therefore a useful tool to classify finitely generated groups up to quasi-isometry. Motivated by Gersten and Cooper-Mihalik's notions, we introduce two types of relative divergence of a finitely generated group with respect to a subgroup: upper relative divergence and lower relative divergence. 

We sketch the idea of relative divergence by the simplified definition and we refer readers Section~ \ref{Concepts} for the exact definition. We first introduce some notations and we will work on them for the concept of relative divergence. Let $(X,d)$ be a geodesic space and $A$ a subspace. For each positive $r$, let $d_{r,A}$ be the induced length metric on the complement of the $r$--neighborhood of $A$ in $X$. We now define the relative divergence of the space $X$ with respect to the subspace $A$ (both upper relative divergence and lower relative divergence). Fix some number $\rho$ in $(0,1]$ and some positive integer $n$.

For each positive $r$, let $\delta(r)=\sup d_{\rho r,A}(x,y)$ where the supremum is taken over all $x, y$ which lie in $\partial N_r(A)$ such that $d_{r,A}(x, y)<\infty$ and $d(x,y)\leq nr$ (see Figure~ \ref{afirst}).

Similarly, let $\sigma(r)=\inf d_{\rho r,A}(x,y)$ where the infimum is taken over all $x, y$ which lie in $\partial N_r(A)$ such that $d_{r,A}(x, y)<\infty$ and $d(x,y)\geq nr$ (see Figure~ \ref{asecond}). 

\begin{figure}

%\begin{center}
\centering
\begin{subfigure}{.5\textwidth}
\centering
\scalebox{.3}{\includegraphics{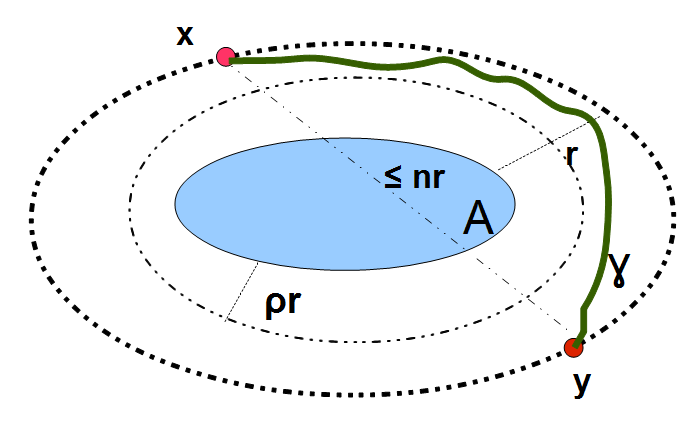}}
\caption{}
\label{afirst}
\end{subfigure}%
\begin{subfigure}{.5\textwidth}
\centering
\scalebox{.3}{\includegraphics{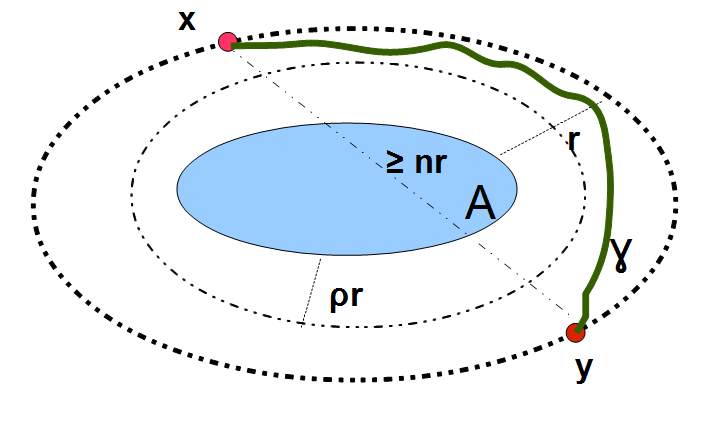}}
\caption{}
\label{asecond}
\end{subfigure}
\caption{The picture illustrates the idea of upper and lower relative divergence of a geodesic space $X$ with respect to a subspace $A$}
\label{a}
%\end{center}
\end{figure}

%\begin{center}
%\scalebox{.3}{\includegraphics{uRel.png}}
%\scalebox{.3}{\includegraphics{lRel.png}}
%\end{center}
%\end{figure}
The function $\delta$ is the upper relative divergence of the pair $(X,A)$, denoted by \emph{$Div(X,A)$}, and the function $\sigma$ is the lower relative divergence of the pair $(X,A)$, denoted by \emph{$div(X,A)$}.

In Section~ \ref{Concepts}, we show that both upper relative divergence and lower relative divergence depend only on the quasi-isometry type of $(X,A)$. Therefore, we can define both the upper and the lower relative divergence of a pair $(G,H)$, denoted by \emph{$Div(G,H)$} and \emph{$div(G,H)$}, where $G$ is a finitely generated group and $H$ is a subgroup. While upper relative divergence generalizes upper divergence introduced by Gersten \cite{MR1254309}, lower relative divergence generalizes lower divergence defined by Cooper-Mihalik \cite{MR1170363}. The relative divergence of a pair $(G,H)$ measures the distance distortion of the complement of the $r$--neighborhood of $H$ in the Cayley graph of $G$ when $r$ increases. 	

\subsection{Upper relative divergence}

The following theorem describes the upper relative divergence of a finitely generated group with respect to a finitely generated normal subgroup. 

\begin{thm}
\label{i1}
Let $G$ be a finitely generated group and $H$ a finitely generated normal subgroup of $G$ such that $G/H$ is one-ended. Then \[Div(G/H,e)\preceq Div(G,H)\preceq Dist^H_G \circ Div(G/H,e)\] where $Dist^H_G$ is the upper distortion of $H$ in $G$.
\end{thm}

In the above theorem, we use the well-known concept of distortion of subgroups. This concept, in some sense, measures the ``upper bound'' of the distance distortion of a subgroup in comparison with the distance of a whole group. However, we realize that we also need the concept of ``lower bound'' of the distance distortion of subgroups to better understand how a subgroup is embedded into a whole group. Therefore, we introduce the concept of lower distortion and we refer to the traditional concept of distortion as upper distortion (see Section~ \ref {sd}).
 
The upper divergence of a one-ended relative hyperbolic group is at least exponential by Sisto \cite{Sisto}. The following theorem strengthens the result of Sisto. 

\begin{thm}
\label{i2}
Let $(G,\PP)$ be a relatively hyperbolic group and $H$ a subgroup of $G$ such that the number of filtered ends of $H$ in $G$ is finite. We assume that $H$ is not conjugate to an infinite index subgroup of any peripheral subgroup. Then $Div(G,H)$ is at least exponential.
\end{thm}
We refer the readers to Section~ \ref{feopog} for the definition of the number of filtered ends.
\subsection{Lower relative divergence}

As mentioned earlier, the lower divergence of a finitely generated group is either linear or exponential. The lower relative divergence of a pair of groups, on the other hand, is more diverse. 

\begin{thm}
\label{i3}
Let $f$ be any polynomial function or exponential function. There is a pair of groups $(G,H)$, where $G$ is a $\CAT(0)$ group (i.e. the group that acts properly and cocompactly on some $\CAT(0)$ space) and $H$ is an infinite cyclic subgroup of $G$, such that 
$div(G,H)$ is $f$.
\end{thm}

In Theorem~ \ref{ldns} and Theorem~ \ref{ldocs}, we compute the lower relative divergence of a pair of groups $(G,H)$ when $H$ is an infinite normal subgroup or an infinite cyclic subgroup. In order to measure the lower relative divergence of a finitely generated group with respect to a normal subgroup, we use the concept of lower distortion of a subgroup (which is mentioned earlier). Although the idea of lower distortion is implicit in works of Gromov \cite {MR1253544}, Ol'shanski \cite{MR1714850} and many others, the exact concept does not seem to be recorded in the literature. When investigating the lower relative divergence of a pair $(G,H)$ in the case $H$ is a cyclic subgroup, we will see the connection between the concept of relative lower divergence and both upper distortion and upper divergence.

We also examine the lower relative divergence of a relatively hyperbolic group with respect to a subgroup. While the upper relative divergence of a finitely generated relatively hyperbolic group with respect to almost all subgroups is at least exponential (See Theorem \ref{i2}), its lower relative divergence can be linear (see Theorem~ \ref{ldonsirhg} and Theorem~ \ref{ldorqcs}). Moreover, we also examine the lower relative divergence of a finitely generated relatively hyperbolic group with respect to a fully relatively quasiconvex subgroup in the following theorem. 

\begin{thm}
\label{i4}
Let $(G,\PP)$ be a relatively hyperbolic group and $H$ an infinite fully relatively quasiconvex subgroup of $G$. If the number of filtered ends of $H$ in $G$ is finite, then $div(G,H)$ is at least exponential.
\end{thm}

In the above theorem, if we drop the condition ``fully relative quasiconvexness'' of the subgroup $H$, the conclusion of the theorem is no longer true (see Theorem \ref{ldorqcs}).
 
\subsection{Overview}

In Section~ \ref{prelim}, we prepare some preliminary knowledge for the main part of the paper. This knowledge will be used to define the concept of relative divergence and compute relative divergence of certain pairs of groups.

In Section~ \ref{sd}, we recall the concept of distortion of a subgroup, which we call upper distortion and introduce the related concept of lower distortion. Together with upper distortion, lower distortion helps us understand the connection between the geometry of a group and the geometry of its subgroups. We also carefully investigate this new concept although it is not the main part of this paper.

In Section~ \ref{Concepts}, we give precise definitions of upper and lower divergence of a pair $(X,A)$, where $X$ is a geodesic space and $A$ is a subspace. We use these concepts to define the upper and lower divergence of a pair $(G,H)$, where $G$ is a finitely generated group and $H$ is a subgroup. We also investigate some key properties of relative divergence.

In Sections~ \ref{rdns} and \ref{rdcs}, we investigate the divergence of a finitely generated group with respect to a normal subgroup or a cyclic subgroup. In Section~ \ref{rdns}, the proof of Theorem~ \ref{i1} is also shown.

In Section~ \ref{rdcat(0)g}, we examine relative divergence of some $CAT(0)$ groups. We also investigate a family of groups studied by Macura \cite{MR3032700} to show that relative lower divergence can be a polynomial function with arbitrary degree. In this section, readers can find the proof of Theorem~ \ref{i3} for the case the lower divergence is polynomial. 

In Section~ \ref{rdrhg}, we examine the relative divergence of a relatively hyperbolic group. We also investigate the lower relative divergence of a relatively hyperbolic group with respect to a fully relatively quasiconvex subgroup and use this fact to show that the lower divergence of a pair of groups can be at least exponential. In this section, we show the proofs of Theorem~ \ref{i2} and Theorem~ \ref{i4}. Moreover, readers can see the proof of Theorem~ \ref{i3} for the case the lower divergence is exponential in this section.
\subsection*{Acknowledgments}
I would like to thank my advisor Prof.~Christopher Hruska for very helpful comments and suggestions.

\section{Preliminaries}
\label{prelim}
In this section, we discuss some preliminary background before discussing the main part of the paper. We first construct the notions of domination and equivalence. We review some concepts in geometric group theory: geodesic spaces, quasigeodesics, quasi-isometry and quasi-isometric embedding, and the number of filtered ends of pairs of groups. We also introduce the concept of quasi-isometry between two pairs of metric spaces. 

\subsection{The notions of domination and equivalence}
In this section, we build the notions of domination and equivalence on the set of some certain families of functions. These notions are the tool to measure the relative divergence of a finitely generated group with respect to a subgroup.

\begin{defn}
Let $\mathcal{M}$ be the collection of all functions from $[0,\infty)$ to $[0,\infty]$. Let $f$ and $g$ be arbitrary elements of $\mathcal{M}$. \emph{The function $f$ is dominated by the function $g$}, denoted \emph{$f\preceq g$}, if there are positive constants $A$, $B$ and $C$ such that $f(x)\leq g(Ax)+Bx$ for all $x>C$. Two function $f$ and $g$ are \emph{equivalent}, denoted \emph{$f\sim g$}, if $f\preceq g$ and $g\preceq f$. \emph{The function $f$ is strictly dominated by the function $g$}, denoted \emph{$f\prec g$}, if $f$ is dominated by $g$ and they are not equivalent.

\end{defn}

\begin{rem}
The relations $\preceq$ and $\prec$ are transitive. The relation $\sim$ is an equivalence relation on the set $\mathcal{M}$.

Let $f$ and $g$ be two polynomial functions in the family $\mathcal{M}$. We observe that $f$ is dominated by $g$ iff the degree of $f$ is less than or equal to the degree of $g$ and they are equivalent iff they have the same degree. All exponential functions of the form $a^{bx+c}$, where $a>1, b>0$ are equivalent. Therefore, a function $f$ in $\mathcal{M}$ is \emph{linear, quadratic or exponential...} if $f$ is respectively equivalent to any polynomial with degree one, two or any function of the form $a^{bx+c}$, where $a>1, b>0$.
\end{rem}

\begin{defn}
Let $\{\delta^n_{\rho}\}$ and $\{\delta'^n_{\rho}\}$ be two families of functions of $\mathcal{M}$, indexed over $\rho \in (0,1]$ and positive integers $n\geq 2$. \emph{The family $\{\delta^n_{\rho}\}$ is dominated by the family $\{\delta'^n_{\rho}\}$}, denoted \emph{$\{\delta^n_{\rho}\}\preceq \{\delta'^n_{\rho}\}$}, if there exists constant $L\in (0,1]$ and a positive integer $M$ such that $\delta^n_{L\rho}\preceq \delta^{Mn}_{\rho}$. The notions of strict domination and equivalence can be defined as above. %Two families \emph{$\{\delta^n_{\rho}\}$ and $\{\delta'^n_{\rho}\}$} are \emph{equivalent}, denoted \emph{$\{\delta^n_{\rho}\}\sim \{\delta'^n_{\rho}\}$}, if $\{\delta^n_{\rho}\}\preceq \{\delta'^n_{\rho}\}$ and $\{\delta'^n_{\rho}\}\preceq \{\delta^n_{\rho}\}$. \emph{The family $\{\delta^n_{\rho}\}$ is strictly dominated by the family $\{\delta'^n_{\rho}\}$}, denoted \emph{$f\prec g$}, if $\{\delta^n_{\rho}\}$ is dominated by $\{\delta'^n_{\rho}\}$ and they are not equivalent.
\end{defn}

\begin{rem}
The relations $\preceq$ and $\prec$ are transitive. The relation $\sim$ is an equivalence relation.

If $f$ is an element in $\mathcal{M}$, we could represent $f$ as a family $\{\delta^n_{\rho}\}$ for which $\delta^n_{\rho}=f$ for all $\rho$ and $n$. Therefore, the family $\{\delta^n_{\rho}\}$ is dominated by (or dominates) a function $f$ in $\mathcal{M}$ if $\{\delta^n_{\rho}\}$ is dominated by (or dominates) the family $\{\delta'^n_{\rho}\}$ where $\delta'^n_{\rho}=f$ for all $\rho$ and $n$. The equivalence between a family $\{\delta^n_{\rho}\}$ and a function $f$ in $\mathcal{M}$ can be defined similarly. Thus, a family $\{\delta^n_{\rho}\}$ is linear, quadratic, exponential, etc if $\{\delta^n_{\rho}\}$ is equivalent to the function $f$ where $f$ is linear, quadratic, exponential, etc.
\end{rem}

\subsection{Geodesic spaces, quasigeodesics, quasi-isometry}

In this section, we review the concepts of geodesic spaces, quasigeodesics, quasi-isometry and quasi-isometric embedding, and we introduce the concept of quasi-isometry between two pair of metric spaces. These concepts play an important role in defining the concept of upper relative divergence and lower relative divergence of a finitely generated group with respect to a subgroup. Most of information in this section is cited from \cite{MR1086648}. 

\begin{rem}
For each path with finite length $\alpha$ in a geodesic space $X$, we denote the endpoints of $\alpha$ by $\alpha_+$, $\alpha_-$ and the length of $\alpha$ by $\ell(\alpha)$. For each ray $\alpha$ in a space $X$, we denote the initial point of $\alpha$ by $\alpha_+$.
\end{rem}
 
\begin{defn}
Let $(X,d)$ be a metric space.
\begin{enumerate}
\item A path $p$ in $X$ is an \emph{$(L,C)$--quasigeodesic} for some $L\geq 1$ and $C\geq 0$, if for every subpath $q$ of $p$ the inequality $\ell(q)\leq L\,d(q_+,q_-)+C$ holds.
\item A path $p$ in $X$ is a \emph{quasigeodesic} if it is $(L,C)$--quasigeodesic for some $L\geq 1$ and $C\geq 0$. 
\item A path $p$ in $X$ is an \emph{L--quasigeodesic} if it is $(L,L)$--quasigeodesic for some $L\geq 1$. 
\item A path $p$ in $X$ is a \emph{geodesic} if it is $(1,0)$--quasigeodesic.
\item Two quasigeodesics are \emph{equivalent} if the Hausdorff distance between their images is finite.
\item The metric space $X$ is a \emph{geodesic space} if any pair of points in $X$ can be joined by a geodesic segment.
\end{enumerate}
\end{defn}

\begin{defn}
Let $(X,d_X)$ and $(Y,d_Y)$ be two metric spaces. The map $\Phi$ from $X$ to $Y$ is \emph{a quasi-isometry} if there is a constant $K\geq 1$ and a function $\Psi$ from $Y$ to $X$ such that the following holds:
\begin{align}
d_Y\bigl(\Phi(x_1),\Phi(x_2)\bigr)&\leq K\,d_X(x_1,x_2)+K \text{ for all } x_1, x_2 \text{ in } X\\
d_X\bigl(\Psi(y_1),\Psi(y_2)\bigr)&\leq K\,d_Y(y_1,y_2)+K \text{ for all } y_1, y_2 \text{ in } Y\\
d_Y\bigl(\Phi\circ\Psi(y), y\bigr)&\leq K \text{ for all } y \text{ in } Y\\
d_X\bigl(\Psi\circ\Phi(x), x\bigr)&\leq K \text{ for all } x \text{ in } X
\end{align}
\end{defn}

The proof of the following lemma is obvious, and we leave it to the reader.

\begin{lem}
\label{qip}
Let $(X,d_X)$ and $(Y,d_Y)$ be two geodesic spaces and the map $\Phi$ from $X$ to $Y$ a quasi-isometry. Then there is a constant $C\geq1$ such that the following hold:
\begin{enumerate}
\item $({1}/{C})\,d_X(x_1,x_2)-1\leq d_Y\bigl(\Phi(x_1),\Phi(x_2)\bigr)\leq C\,d_X(x_1,x_2)+C$, for all $x_1, x_2$ in $X$
\item $N_C\bigl(\Phi(X)\bigr)=Y$
\item If $\alpha$ is a path connecting two points $x_1$ and $x_2$ in $X$, then there is a path $\beta$ connecting $\Phi(x_1)$ and $\Phi(x_2)$ in $Y$ such that the Hausdorff distance between $\Phi(\alpha)$ and $\beta$ is at most $C$. Moreover, $\abs{\beta}\leq C\abs{\alpha}+C$.
\item If $\beta$ is a path connecting two points $\Phi(x_1)$ and $\Phi(x_2)$ for some $x_1, x_2 \in X$, then there is a path $\alpha$ connecting $x_1$ and $x_2$ in $X$ such that the Hausdorff distance between $\Phi(\alpha)$ and $\beta$ is at most $C$. Moreover, $\abs{\alpha}\leq C\abs{\beta}+C$.
\end{enumerate} 
\end{lem}

\begin{defn}
Let $(X,d_X)$ and $(Y,d_Y)$ be two geodesic spaces and the map $\Phi$ from $X$ to $Y$ a quasi-isometric embedding if \[({1}/{C})\,d_X(x_1,x_2)-1\leq d_Y\bigl(\Phi(x_1),\Phi(x_2)\bigr)\leq C\,d_X(x_1,x_2)+C\] for all $x_1, x_2$ in $X$.
\end{defn}

\begin{rem}
Throughout this paper, we denote $(X,A)$ to be a pair of metric spaces, where $X$ is a geodesic space and $A$ is a subspace of $X$. 
\end{rem}
\begin{defn}
Two pairs of spaces $(X,A)$ and $(Y,B)$ are \emph{quasi-isometric} if there is a quasi-isometry $\Phi$ from $X$ to $Y$ such that the Hausdorff distance between $\Phi(A)$ and $B$ is finite.
\end{defn}

It is not hard to prove the following proposition and we leave it to the reader.
\begin{prop}
Quasi-isometry of pairs of metric spaces is an equivalence relation.
\end{prop}

\subsection{Filtered ends of pairs of groups}
\label{feopog}
In this section, we review the concepts of the number of ends of groups and the number of filtered ends of pairs of groups. We refer the readers to Chapter~ 14 in \cite{MR2365352} for the proof of all the statements on these concepts. We also prove the lemma on the existence of subgroup perpendicular ray which is defined below.

We now define the concept of the number of filtered ends of a pair of groups and we will see that this concept generalizes the concept of the number of ends of a group. 
\begin{defn}
Let $G$ be a group with a finite generating set $S$ and $H$ a subgroup of $G$. For each positive $r$ a connected component $U$ of $C_r(H)$ in the Cayley graph $\Gamma(G,S)$ is \emph{deep} if $U$ does not lie in the $s$--neighborhood of $H$ for any positive s. Let $\tilde{e}_r(G,H)$ be the number of deep components of $C_r(H)$. We note that $\tilde{e}_r(G,H)\geq \tilde{e}_s(G,H)$ if $r>s$. The number of \emph{filtered ends of the pair} $(G,H)$, denoted $\tilde{e}(G,H)$, is the supremum of the set $\bigset{\tilde{e}_r(G,H)}{r>0}$. 
\end{defn}

\begin{rem}
Let $G$ be a finitely generated group and $H$ a subgroup.
\begin{enumerate}
\item The number $\tilde{e}(G,H)$ does not depend on the choice of finite generating set $S$ of $G$ and $\tilde{e}(G,H)=0$ iff $H$ is a finite index subgroup of $G$.

\item If $\tilde{e}(G,H)=m<\infty$, then there is a positive number $r_0$ such that $C_r(H)$ has exactly $m$ deep components for each $r>r_0$.

\item When $H$ is the trivial subgroup, $\tilde{e}(G,H)$ is \emph{the number of ends of $G$}, denoted $\tilde{e}(G)$. A finitely generated group is \emph{one-ended} if $\tilde{e}(G)=1$ 
\end{enumerate}

\end{rem}

\begin{thm}[Proposition 14.5.9, \cite {MR2365352}]
\label{feons}
If $H$ is a finitely generated normal subgroup of $G$ then $\tilde{e}(G,H)$ equals the number of ends of $G/H$.
\end{thm}

\begin{defn}
Let $G$ be a group with a finite generating set $S$ and $H$ an infinite index subgroup of $G$. A geodesic ray $\gamma$ in the Cayley graph $\Gamma(G,S)$ is \emph{$H$--perpendicular} if the initial point $h$ of $\gamma$ lies in $H$ and $d_S(\gamma(r),H)=r$ for all positive $r$.
\end{defn}

The following lemma shows the existence of many $H$--perpendicular geodesic rays.

\begin{lem}
\label{pr}
Let $G$ be a group with a finite generating set $S$ and $H$ an infinite index subgroup of $G$. Then for each element $h$ in $H$, there is an $H$--perpendicular geodesic ray with the initial point $h$.
\end{lem}
\begin{proof}
For each positive integer $n$, there is a vertex $g_n$ in $C_n(H)$. Let $k_n$ be an element in $H$ and $\alpha_n$ a geodesic segment connecting $g_n$ and $k_n$ such that the length of $\alpha_n$ is equal to the distance between $g_n$ and $H$. We define $\gamma_n=(hk^{-1}_n)\alpha_n$, then $\gamma_n$ is a geodesic segment with the initial point $h$ and $d_S\bigl(\gamma_n(r),H\bigr)=r$ for all positive $r$ less than the length of $\gamma_n$. By the Arzela-Ascoli theorem, there is a geodesic ray $\gamma$ with the initial point $h$ such that $d_S\bigl(\gamma(r),H\bigr)=r$ for all positive $r$. \qedhere
\end{proof} 

\section{Distortion of subgroups}
\label{sd}
In this section, we will review the concept of distortion of a subgroup, which we call upper distortion. This concept of distortion will later help us compute relative divergence of a large class of pairs of groups. We also introduce the concept of lower distortion of a subgroup. This new concept is also a tool to compute relative divergence. We investigate some key properties of lower distortion and the relation between lower distortion and upper distortion. 

First of all, we will review the concept of upper distortion.
\begin{defn}
Let $G$ be a group with a finite generating set $S$ and $H$ a subgroup of $G$ with a finite generating set $T$. 
The \emph{ upper subgroup distortion} of $H$ in $G$ is the function is the function $Dist^H_G\!:(0,\infty)\to(0,\infty)$ defined as follows:
\[Dist^H_G(r)=\max \bigset{\abs{h}_T}{h\in H, \abs{h}_S\leq r}.\] 
%\item The \emph{the lower subgroup distortion} of $H$ in $G$ is a function, denoted $dist^H_G$, from $(0,\infty)$ to itself defined as follows:
%\[dist^H_G(r)=\min \set{\abs{h}_T}{h\in H, \abs{h}_S\geq r}.\] 
\end{defn}

\begin{rem}
It is well-known that the concept of upper distortion does not depend on the choice of finite generating sets $S$ and $T$. More precisely, the functions $Dist^H_G$ are equivalent for all pairs of finite sets $(S,T)$ generating $(G,H)$ respectively.

The function $Dist^H_G$ is non-decreasing, and dominates a linear function.

A finitely generated subgroup $H$ of $G$ is \emph{undistorted} if $Dist^H_G$ is linear.
\end{rem}

We now introduce the concept of lower distortion.
\begin{defn}
Let $G$ be a group with a finite generating set $S$ and $H$ a subgroup of $G$ with a finite generating set $T$. The \emph{lower distortion} of $H$ in $G$ is the function $dist^H_G\!:(0,\infty)\to(0,\infty)$ defined as follows:
\[dist^H_G(r)=\min \bigset{\abs{h}_T}{h\in H, \abs{h}_S\geq r}.\] 
We use the convention that the minimum of the empty set is 0.
\end{defn}

\begin{rem}
Similar to the concept of upper distortion, the concept of lower distortion also does not depend on the choice of generating sets. When $H$ is an infinite subgroup, the function $dist^H_G$ is non-decreasing and dominates a linear function. 
\end{rem}

The following proposition shows a relation between upper distortion and lower distortion.
\begin{prop}
Let $G$ be a finitely generated group and $H$ a finitely generated subgroup of $G$. Then $dist^H_G \preceq Dist^H_G$.
\end{prop}
\begin{proof}
Let $S$ be a finite generating set of $G$ and we assume that $S$ contains the finite generating set $T$ of the subgroup $H$. Thus, we could consider $\Gamma(H,T)$ as a subgraph of $\Gamma(G,S)$. If $H$ is a finite subgroup then $dist^H_G$ is a bounded function and the proof follows easily. Thus, we assume $H$ is an infinite subgroup. 

For each $r>1$, we could chose an element $k$ in $H$ such that $\abs{k}_{S}\geq 2r$. We connect the identity element $e$ and $k$ by a geodesic $\alpha$ in $\Gamma(H,T)$. Thus, we can choose $h$ be an element in $\alpha$ such that 
$r\leq\abs{h}_{S}\leq 2r$. Since $h$ is also an element of $H$, then $dist^H_G(r) \leq \abs{h}_{T} \leq Dist^H_G(2r)$. Thus, $dist^H_G \preceq Dist^H_G$. \qedhere
\end{proof}

We now investigate some key properties of lower distortion:

\begin{thm}
Suppose that $G$, $H$, $K$ are all infinite finitely generated groups and $K\leq H \leq G$. Then:
\begin{enumerate}
\item $dist^K_H \circ dist^H_G \preceq dist^K_G$
\item $dist^K_H \preceq dist^K_G$
\item $dist^H_G \preceq dist^K_G$
\item If $\abs{G:H}< \infty$, then $dist^K_G \sim dist^K_H$
\item If $\abs{H:K}< \infty$, then $dist^K_G \sim dist^H_G$
\item If $H_1$ and $H_2$ are two commensurable finitely generated subgroups, then $dist^{H_1}_G \sim dist^{H_2}_G$
\end{enumerate}
\end{thm}

\begin{proof}
We call $S_1$, $S_2$ and $S_3$ finite generating sets of $G$, $H$ and $K$ respectively. We can assume that $S_3 \subset S_2 \subset S_1$. We now prove that
\[dist^K_H \circ dist^H_G(n) \leq dist^K_G(n) \text{ for all $n$}.\]

For any positive number $n$, we choose $k_0 \in K$ such that $\abs{k_0}_{S_1}\geq n$ and $\abs{k_0}_{S_3}=dist^K_G(n)$. Since $k_0 \in H$ and $\abs{k_0}_{S_1}\geq n$, then $\abs{k_0}_{S_2}\geq dist^H_G(n)$. Therefore, $\abs{k_0}_{S_3}\geq dist^K_H \bigl (dist^H_G(n)\bigr)$. Thus, \[dist^K_H \circ dist^H_G(n) \leq dist^K_G(n) \text{ for all $n$}.\] 

Statements (2) and (3) are immediate results of (1) since the lower distortion functions $dist^H_G$ and $dist^K_H$ are non-decreasing and at least linear.

We now prove Statement (4). Since $dist^K_H \preceq dist^K_G$, then we only need to prove $dist^K_G\preceq dist^K_H$. Since $\abs{G:H}< \infty$, then there is a positive integer $C$ such that \[d_{S_2}(h_1,h_2)\leq C\,d_{S_1}(h_1,h_2)+C \text{ for all $h_1$ and $h_2$ in $H$}.\] We now prove that
\[dist^K_G(n) \leq dist^K_H(2Cn) \text{ for all $n$}.\] 

For any positive number $n>1$, we choose $k_0 \in K$ such that $\abs{k_0}_{S_2}\geq 2Cn$ and $\abs{k_0}_{S_3}=dist^K_H(2Cn)$. Thus, \[\abs{k_0}_{S_1}\geq \frac{\abs{k_0}_{S_2}-C}{C}\geq 2n-1\geq n.\]

Therefore, $dist^K_G(n) \leq dist^K_H(2Cn)$. In particular, $dist^K_G\preceq dist^K_H$.

We now prove Statement (5). Since $dist^H_G \preceq dist^K_G$, then we only need to prove $dist^K_G\preceq dist^H_G$. Since $\abs{H:K}< \infty$, then there is a positive integer $C$ such that \[d_{S_3}(k_1,k_2)\leq C\,d_{S_2}(k_1,k_2)+C \text{ for all $k_1$ and $k_2$ in $K$},\] and $H\subset N_C(K)$ with respect to metric $d_{S_2}$. We now show that 
\[dist^K_G(n) \leq Cdist^H_G(2n)+C^2+C \text{ for all $n\geq C$}.\]

For any positive number $n$ greater than $C$, we choose $h_0 \in H$ such that $\abs{h_0}_{S_1}\geq 2n$ and $\abs{h_0}_{S_2}=dist^H_G(2n)$. Since $H\subset N_C(K)$ with respect to metric $d_{S_2}$, then there is $k_0\in K$ such that $d_{S_2}(k_0,h_0)\leq C$. In particular, $d_{S_1}(k_0,h_0)\leq C$. Thus, \[\abs{k_0}_{S_1}\geq \abs{h_0}_{S_1}-C\geq 2n-C\geq n.\]

Thus, $\abs{k_0}_{S_3}\geq dist^K_G(n)$

Also \[\abs{k_0}_{S_3}\leq C\abs{k_0}_{S_2}+C\leq C(\abs{h_0}_{S_2}+C)+C\]

and \[\abs{h_0}_{S_2}=dist^H_G(2n)\]

Therefore, $dist^K_G(n) \leq Cdist^H_G(2n)+C^2+C$. In particular, $dist^K_G\preceq dist^H_G$. 

We easily obtain (6) from (5) by observing that $\abs{H_1:(H_1\cap H_2)}<\infty$ and $\abs{H_2:(H_1\cap H_2)}<\infty$. \qedhere
\end{proof}

We now explain the relationship between the lower distortion and the growth of a finitely generated group. We will see that the growth function will be an upper bound of the lower distortion. Before showing this fact, we need to review the concept of growth of groups.

\begin{defn}
Let $G$ be a group with a finite set of generators $S$. \emph{The growth} of $G$, denoted by \emph{$Growth_G$}, is a function $f\!:[0,\infty)\to [0,\infty)$ to itself defined by letting $f(r)$ to be the number of elements of $G$ that lie in the ball $B(e,r)$ for each $r\geq 0$.
\end{defn}

\begin{rem}
It is well-known that the growth of a finitely generated group does not depend on the choice of finite generating set (the proof is almost identical to the case of upper distortion). More precisely, the functions $Growth_G$ are equivalent for all finite sets $S$ of generators of $G$. Moreover, the function $Growth_G$ is dominated by the exponential function. 
\end{rem}

\begin{prop}
\label{ldidbgf}
Let $G$ be a finitely generated group and $H$ a finitely generated subgroup of $G$. Then the lower distortion $dist^H_G$ is dominated by the growth function $Growth_G$ of $G$. In particular, the lower distortion $dist^H_G$ is dominated by the exponential function.
\end{prop}
\begin{proof}
Let $S$ be a finite generating set of $G$. We will assume that $S$ contains the finite generating set $T$ of the subgroup $H$. Thus, we could consider $\Gamma(H,T)$ as a subgraph of $\Gamma(G,S)$.
If $H$ is finite, then $dist^H_G$ is bounded and the proof follows easily. Thus, we assume $H$ is an infinite subgroup. 

For each $r>1$, we could chose an element $h$ in $H$ such that $\abs{h}_{S}\geq r$. We connect the identity element $e$ and $h$ by a geodesic $\alpha$ in $\Gamma(H,T)$. Let $h'$ be a vertex in $\alpha$ such that $\abs{h'}_{S}\geq r$ and the subpath $\alpha'$ of $\alpha$ connecting $e$ and $h'$ must lie in the closed ball with center $e$ and radius $2r$ of $\Gamma(G,S)$. Thus, the length of $\alpha'$ is bounded by the number of vertices in this ball. Therefore, $\abs{h'}_{T}$ is bounded by the number of vertices of the closed ball with center $e$ and radius $2r$ in $\Gamma(G,S)$. Thus, $dist^H_G(r)\leq Growth_G(2r)$. Therefore, $dist^H_G\preceq Growth_G$. \qedhere
\end{proof}

We now find some examples of finitely generated groups and its finitely generated subgroups to see their lower distortion. The following theorem can be deduced from the work of Milnor (see the proof of Lemma 4 in \cite {MR0232311}). We just use the new concept of lower distortion to interpret a part of Milnor's work. 

\begin{thm}
\label{rd}
Let $G=\langle a,b,c| bab^{-1}a^{-1}=c, ac=ca, bc=cb\rangle$ be the Heisenberg group and $H$ the cyclic group generated by $c$. Then $dist^H_G$ and $Dist^H_G$ are both quadratic.
\end{thm}

\begin{rem}
In \cite{MR647496}, Tits investigates the growth of a finitely generated virtually nilpotent group. We can use a part of his work to find a pair $(G,H)$, where $G$ is a finitely generated nilpotent group and $H$ is a finitely generated subgroup, such that $dist^H_G$ and $Dist^H_G$ can be equivalent to the same polynomial with arbitrary degree.

In \cite{MR1872804}, Osin also gives a formula to compute upper distortion of arbitrary subgroups of nilpotent groups.
\end{rem}

Before studying more examples about lower distortion, we need to review the concept of length functions and a key theorem. 

\begin{defn}
Let $G$ be a group with a finite generating set $S$ and $H$ a subgroup of $G$. The \emph{length function} $\ell$ of $H$ inside $G$ is the function from the group $H$ to the set of natural numbers as follows: \[\ell(h)=\abs{h}_S \text{ for $h\in H$ }.\]
\end{defn}

 \begin{rem}
In some sense, the concept of length function can give us more information than the concepts of upper and lower distortion when we investigate an embedding of a subgroup.
\end{rem}

\begin{thm}(\cite{MR1714850})
\label{lengthfunction}
Let $\ell$ be the length function of group $H$ inside some finitely generated group $G$. Then the following conditions hold:
\begin{enumerate}
\item $\ell(h)=\ell(h^{-1})$ for every $h\in H$; $\ell(h)=0$ iff $h=e$.
\item $\ell(h_1h_2)\leq \ell(h_1)+\ell(h_2)$ for every $h_1, h_2 \in H$.
\item There is a positive integer $C$ such that the cardinality of the set $\bigset{h\in H}{\ell(h)\leq n}$ does not exceed $C^n$ for every natural number $n$
\end{enumerate}
Conversely for every group $H$ and every function $\ell$ from $H$ to the set of natural numbers satisfying (1)--(3), there exists an embedding of $H$ into a 2--generated group $G$ with a finite generating set $S=\{g_1,g_2\}$ such that the length function $\ell_1$ of $H$ inside $G$ is equivalent to $\ell$ (i.e. there exists a positive integer $B$ such that $({1}/{B})\ell(h)\leq \ell_1(h)\leq B\ell(h)$).

\end{thm}

\begin{defn}
A function $f\!:[0,\infty)\to [0,\infty)$ is subadditive if $f(i+j)\leq f(i)+f(j)$ for every positive numbers $i$ and $j$. 
\end{defn}

We now apply Theorem~ \ref{lengthfunction} to show that any finitely generated group $H$ can be a subgroup of a finitely generated group $G$ such that the lower distortion and the upper distortion of $H$ in $G$ can be both equivalent to any element of some large class of functions.

\begin{thm}
Let $f\!:[0,\infty)\to [0,\infty)$ be a strictly increasing function such that $f(0)=0$ and $f^{-1}$ is subadditive. Suppose that there is a positive integer $C$ such that $f(n)\leq C^n$ for every positive $n$. Let $H$ be a finitely generated group such that its growth is bounded by some polynomial function. Then there is a finitely generated group $G$ such that $dist^H_G\sim Dist^H_G\sim f$.
\end{thm}

\begin{proof}
We fix some finite generating set $T$ for $H$. Let $A$ and $m$ be a positive integers such that the number of group elements in a ball with radius $n$ is bounded by $An^m$ for every positive integer $n$. For each nonnegative number $x$, we define $\lceil x\rceil$ to be the smallest integer that is greater than or equal to $x$. We now define the length function $\ell\!:H\to \NN$ as follows: \[\ell(h)=\bigl\lceil f^{-1}\bigl(\abs{h}_{T}\bigr)\bigr\rceil \text{ for every $h\in H$ }.\]

We will check $\ell$ satisfies Conditions (1)--(3) in Theorem~ \ref{lengthfunction}. Obviously, $\ell(h)=\ell(h^{-1})$ for every $h\in H$ and $\ell(h)=0$ iff $h=e$. We now check $\ell$ satisfies Condition (2). Indeed, for every $h_1, h_2\in H$ 
\begin{align*}
\ell(h_1h_2)&=\bigl\lceil f^{-1}\bigl(\abs{h_1h_2}_{T}\bigr)\bigr\rceil \\&\leq \bigl\lceil f^{-1}\bigl(\abs{h_1}_{T}+\abs{h_2}_{T}\bigr)\bigr\rceil \\&\leq \bigl\lceil f^{-1}\bigl(\abs{h_1}_{T}\bigr)+f^{-1}\bigl(\abs{h_2}_{T}\bigr)\bigr\rceil\\&\leq \bigl\lceil f^{-1}\bigl(\abs{h_1}_{T}\bigr)\bigr\rceil+\bigl\lceil f^{-1}\bigl(\abs{h_2}_{T}\bigr)\bigr\rceil\\&\leq \ell(h_1)+\ell(h_2).
\end{align*}

Finally, we check $\ell$ satisfies Condition (3). Since for each nonnegative integer $n$ \begin{align*} \bigset{h\in H}{\ell(h)\leq n}&=\bigset{h\in H}{\lceil f^{-1}(\abs{h}_{T})\rceil\leq n}\\&=\bigset{h\in H}{f^{-1}(\abs{h}_{T})\leq n}\\&=\bigset{h\in H}{\abs{h}_{T}\leq f(n)}\\&\subset \bigset{h\in H}{\abs{h}_{T}\leq C^n} \end{align*}
and the cardinality of the set $\bigset{h\in H}{\abs{h}_{T}\leq C^n}$ is bounded by $A {(C^m)}^n$, then the cardinality of the set $\bigset{h\in H}{\ell(h)\leq n}$ is bounded by $A {(C^m)}^n$.

By Theorem~ \ref{lengthfunction}, the group $H$ is a subgroup of some finitely generated group $G$ with a finite generating set $S$ such that the function $\ell$ is equivalent to $\ell_1$, where $\ell_1(h)= \abs{h}_{S}$ for every $h\in H$. Therefore, there is a positive integer $B$ such that $({1}/{B})\ell(h)\leq \ell_1(h)\leq B\ell(h)$ for every $h\in H$.

We now show that the upper distortion $Dist^H_G$ is dominated by $f$. For each positive number $n$ and any $h\in H$ such that $\abs{h}_S\leq n$, we see that \[f^{-1}(\abs{h}_{T})\leq \ell(h)\leq B \ell_1(h)\leq Bn.\]
Thus, $\abs{h}_{T}\leq f(Bn)$. Therefore, $Dist^H_G(n)\leq f(Bn)$. In particular, the upper distortion $Dist^H_G$ is dominated by $f$.

We finish the proof of the theorem by showing that the lower distortion $dist^H_G$ dominates $f$. For each positive number $n$ and any $h\in H$ such that $\abs{h}_S\geq Bn+B$, we see that \[f^{-1}(\abs{h}_{T})\geq \ell(h)-1\geq \frac{1}{B} \ell_1(h)-1\geq n.\]
Thus, $\abs{h}_{T}\geq f(n)$. Therefore, $dist^H_G(Bn+B)\geq f(n)$. In particular, the lower distortion $dist^H_G$ dominates $f$. \qedhere
\end{proof}

We know show one pair of groups $(G,H)$ such that $dist^H_G$ and $Dist^H_G$ are not equivalent. The following example is defined by Gromov \cite {MR1253544} 

\begin{exmp}
Let $G=\langle a,b,c| bab^{-1}=a^2, cbc^{-1}=b^2\rangle$ and let $H$ the cyclic subgroup generated $a$. Observe that
\[a^{2^{2^n}}=b^{2^n}ab^{-2^n}=c^nbc^{-n}ac^nb^{-1}c^{-n}\]
Thus, $Dist^H_G(4n+2)\geq 2^{2^n}$ for each positive number $n$. Therefore, the upper distortion $Dist^H_G$ is super-exponential. However, the lower distortion $dist^H_G$ is at most exponential by Proposition~ \ref{ldidbgf}. Therefore, two functions $dist^H_G$ and $Dist^H_G$ are not equivalent.
\end{exmp}

\section{Relative divergence of geodesic spaces and finitely generated groups}
\label{Concepts}
\subsection{Relative upper divergence}
\label{Rud}

In this section, we introduce the concept of relative upper divergence of geodesic spaces as well as finitely generated groups. We also prove that upper relative divergence is a quasi-isometry invariant.

\begin{defn}
Let $X$ be a geodesic space and $A$ a subspace of $X$. Let $r$ be any positive number.
\begin{enumerate}
\item $N_r(A)=\bigset{x \in X}{d_X(x, A)<r}$
\item $\partial N_r(A)=\bigset{x \in X}{d_X(x, A)=r}$ 
\item $C_r(A)=X-N_r(A)$.
\item Let $d_{r,A}$ be the induced length metric on the complement of the $r$--neighborhood of $A$ in $X$. If the subspace $A$ is clear from context, we can use the notation $d_r$ instead of using $d_{r,A}$. 
\end{enumerate}
\end{defn}

\begin{defn}
Let $(X,A)$ be a pair of metric spaces. For each $\rho \in (0,1]$ and positive integer $n\geq 2$, we define a function $\delta^n_{\rho}\!:[0, \infty)\to [0, \infty]$ as follows: 

For each $r$, let $\delta^n_{\rho}(r)=\sup d_{\rho r}(x_1,x_2)$ where the supremum is taken over all $x_1, x_2 \in \partial N_r(A)$ such that $d_r(x_1, x_2)<\infty$ and $d(x_1,x_2)\leq nr$. 
%\[\delta^n_{\rho}(r)=\sup\set{d_{C_{\rho r}(A)}(x_1, x_2)}{\text{$x_1$, $x_2$ $\in \partial N_r(A)$, $d_{C_r(A)}(x_1, x_2)< \infty$, and $d_S(x_1,x_2)\leq nr$}}.\]

%\[\delta^n_{\rho}(r)=\sup_{\substack{x_1,x_2\in \partial N_r(A)\\d_r(x_1, x_2)< \infty\\d_S(x_1,x_2)\leq nr}} d_{\rho r}(x_1, x_2).\]

The family of functions $\{\delta^n_{\rho}\}$ is \emph{the relative upper divergence} of $X$ with respect $A$, denoted $Div(X,A)$.
\end{defn}
Before defining the upper relative divergence of a finitely generated group with respect to a subgroup, we need the following proposition.

\begin{prop}
\label{qiv}
If two pairs of spaces $(X,A)$ and $(Y,B)$ are quasi-isometric, then $Div(X,A)\sim Div(Y,B)$.
\end{prop}

Before proving the above proposition, we need the following lemmas.
\begin{lem}
\label{qiv1}
Let $X$, $Y$ be geodesic spaces and $A$ a subspace of $X$. Let $\Phi$ be a quasi-isometry from $X$ to $Y$. Then $Div(X,A) \preceq Div\bigl(Y,\Phi(A)\bigr)$.
\end{lem}
\begin{proof}
Let $B=\Phi(A)$. Let $Div(X,A) = \{\delta^n_{\rho}\}$ and $Div(Y,B)= \{\delta'^n_{\rho}\}$. Let $K$ be the number provided by Lemma~ \ref{qip}. Let $L={1}/{8K^2}$ and $M=\big[2K(2K+1)+1\big]+1$. We will prove that $\delta^n_{L\rho} \preceq \delta'^{Mn}_{\rho}$. More precisely, we define $r_0=3K(1+K)+{8K^2}/{\rho}$ and we are going to show \[\delta^n_{L\rho}(r)\leq K\delta'^{Mn}_{\rho}\Big(\frac{r}{2K}\Big)+(2K^2+1)r.\]

Indeed, let $x_1$ and $x_2$ be arbitrary points in $\partial N_r(A)$ such that $d_X(x_1,x_2)\leq nr$ and $d_{r,A}(x_1, x_2)< \infty$. Thus, there is a path $\alpha$ in $C_r(A)$ connecting $x_1$ and $x_2$. By Lemma~ \ref{qip}, there is a path $\beta$ connecting $\Phi(x_1)$, $\Phi(x_2)$ such that the Hausdorff distance between $\Phi(\alpha)$ and $\beta$ is at most $K$. Thus, \begin{align*} d_Y\bigl(\beta, B\bigr)&\geq d_Y\bigl(\Phi(\alpha), B\bigr)-K\\&\geq \frac{1}{K}\,d_X\bigl(\alpha, A\bigr)-1-K\\&\geq \frac{r}{K}-1-K\geq\frac{r}{2K} \end{align*} 

Thus, we could choose $y_1$ in $\partial N_{{r}/{2K}}(B)$ and a geodesic $\beta_1$ in $C_{{r}/{2K}}(B)$ connecting $\Phi(x_1)$ and $y_1$ such that the length of $\beta_1$ is bounded above by the distance between $\Phi(x_1)$ and $B$. Also, $d_Y\bigl(\Phi(x_1), B\bigr)\leq K\,d_X(x_1,A)+K\leq Kr+K$. Therefore, the length of $\beta_1$ is at most $Kr+K$. Similarly, we could choose $y_2$ in $\partial N_{{r}/{2K}}(B)$ and a geodesic $\beta_2$ in $C_{{r}/{2K}}(B)$ connecting $\Phi(x_2)$ and $y_2$ such that the length of $\beta_2$ is bounded above by $Kr+K$.

We define $\beta_3=\beta_1\cup \beta\cup\beta_2$, then $\beta_3$ is a path in $C_{{r}/{2K}}(B)$ connecting $y_1$ and $y_2$. Thus, $d_{{r}/{2K},B}(y_1,y_2) < \infty$

Also \begin{align*} d_Y(y_1,y_2)&\leq d_Y\bigl(y_1, \Phi(x_1)\bigr)+d_Y\bigl(\Phi(x_1),\Phi(x_2)\bigr)+d_Y\bigl(\Phi(x_2),y_2\bigr)\\&\leq (Kr+K)+ \bigl(K\,d_X(x_1,x_2)+K\bigr)+(Kr+K)\\&\leq 2Kr+3K+Knr\leq (2K+1)nr\leq Mn(\frac{r}{2K})\end{align*}

We are now going to show that \[d_{L\rho r,A}(x_1,x_2)\leq Kd_{\rho ({r}/{2K}),B}(y_1,y_2)+(2K^2+1)r.\]

Indeed, let $\beta'$ be an arbitrary path in $C_{\rho ({r}/{2K})}(B)$ connecting $y_1$ and $y_2$. We define $\gamma=\beta_1\cup \beta'\cup\beta_2$, then $\gamma$ is a path in $C_{\rho ({r}/{2K})}(B)$ connecting $\Phi(x_1)$, $\Phi(x_2)$ and the length of $\gamma$ is bounded above by $2Kr+2K+\abs{\beta'}$. 

By Lemma~ \ref{qip}, there is a path $\alpha'$ connecting $x_1$ and $x_2$ in $X$ such that the Hausdorff distance between $\Phi(\alpha')$ and $\gamma$ is at most $K$. Moreover, $\abs{\alpha'}\leq K\abs{\gamma}+K$. Since \begin{align*} d_Y\bigl(\Phi(\alpha'), B\bigr)&\geq d_Y(\gamma, B)-K\\&\geq \frac{\rho r}{2K}-K\geq\frac{\rho r}{4K} \end{align*} 

then \begin{align*} d_X(\alpha',A)&\geq \frac{1}{K}\,d_Y\bigl(\Phi(\alpha'), B\bigr)-1\\&\geq \frac{\rho r}{4K^2}-1 \geq\frac{\rho r}{8K^2}\geq L\rho r \end{align*} 
Thus, $\alpha'$ is a path in $C_{L\rho r}(A)$ connecting $x_1$ and $x_2$. Therefore, the distance in $C_{L\rho r}(A)$ between $x_1$ and $x_2$ is bounded above by the length of $\alpha'$. 

Also \begin{align*} \abs{\alpha'}&\leq K\abs{\gamma}+K\\&\leq K\Big(2Kr+2K+\abs{\beta'}\Big)+K\\&\leq K\abs{\beta'}+(2K^2+1)r, \end{align*} and $\beta'$ is an arbitrary path in $C_{\rho ({r}/{2K})}(B)$ connecting $y_1$ and $y_2$.

Thus, \[d_{L\rho r,A}(x_1,x_2)\leq Kd_{\rho ({r}/{2K}),B}(y_1,y_2)+(2K^2+1)r.\]

Therefore, \[\delta^n_{L\rho}(r)\leq K\delta'^{Mn}_{\rho}\Big(\frac{r}{2K}\Big)+(2K^2+1)r.\] 

Thus, $\delta^n_{L\rho} \preceq \delta'^{Mn}_{\rho}$. \qedhere
\end{proof}

\begin{lem} 
\label{qiv2}
Let $X$ be a geodesic space. Let $A$ and $B$ be two subspaces such that the Hausdorff distance between them is finite. Then $Div(X,A)\sim Div(X,B)$.
\end{lem}
\begin{proof}
We only need to prove $Div(X,A) \preceq Div(X,B)$ since the argument for the other direction is almost identical. There is a positive number $r_0$ such that $A$ lies in the $r_0$--neighborhood of $B$ and $B$ also lies in the $r_0$--neighborhood of $A$. Thus, $\bar{N_r(A)}\subset N_{r+r_0}(B)$ and $\bar{N_r(B)}\subset N_{r+r_0}(A)$ for each positive $r$. Let $Div(X,A) = \{\delta^n_{\rho}\}$ and $Div(X,B)= \{\delta'^n_{\rho}\}$. We will to show $\delta^n_{\rho/4} \preceq \delta'^{6n}_{\rho}$. More precisely, we are going to prove that for each $r>{4r_0}/{\rho}$ \[\delta^n_{\rho/4}(r)\leq \delta'^{6n}_{\rho}\Big(\frac{r}{2}\Big)+ 4r.\]

Let $x_1$, $x_2$ be arbitrary points in $\partial N_r(A)$ such that $d_X(x_1,x_2)\leq nr$ and $d_{r,A}(x_1, x_2)< \infty$. Thus, there is a path $\alpha$ in $C_r(A)$ connecting $x_1$ and $x_2$. Therefore, $\alpha$ lies in $C_{r-r_0}(B)$. Thus, $\alpha$ also lies in $C_{r/2}(B)$ because $r/2>r_0$. Moreover, $x_1$ and $x_2$ lies in $N_{r+r_0}(B)$. Therefore, we could choose $y_1$, $y_2$ in $\partial N_{r/2}(B)$ and two geodesics $\beta_1$, $\beta_2$ in $C_{r/2}(B)$ connecting $x_1$, $y_1$ and $x_2$, $y_2$ respectively such that the length of $\beta_1$ and $\beta_2$ are at most $r+r_0$. Since the distance between $x_1$ and $x_2$ is bounded above by $nr$, then the distance between $y_1$ and $y_2$ is at most $nr+2r+2r_0$. Thus, $d_X(y_1,y_2)\leq (n+4)r\leq 3nr\leq 6n(r/2)$. We define $\alpha'=\beta_1\cup \alpha\cup\beta_2$, then $\alpha'$ is a path in $C_{r/2}(B)$ connecting $y_1$ and $y_2$. Thus, $d_{r/2,B}(y_1,y_2) < \infty$.

We are now going to show that \[d_{\rho r/4,A}(x_1,x_2)\leq d_{\rho ({r}/{2}),B}(y_1,y_2)+4r.\]
 
Indeed, let $\gamma$ be an arbitrary path in $C_{\rho ({r}/{2})}(B)$ connecting $y_1$ and $y_2$. Then $\gamma$ also lies in $C_{\rho ({r}/{2})-r_0}(A)$. Therefore, $\gamma$ lies in $C_{{\rho r}/{4}}(A)$. Since $\beta_1$ and $\beta_2$ lies in $C_{r/2}(B)$, then they also lies in $C_{r/2-r_0}(A)$. Thus, $\beta_1$ and $\beta_2$ lies in $C_{{\rho r}/{4}}(A)$. We define $\gamma'=\beta_1\cup \gamma\cup\beta_2$, then $\gamma'$ is a path in $C_{\rho r/4}(A)$ connecting $x_1$ and $x_2$. Thus, $d_{\rho r/4,A}(x_1,x_2)\leq \abs{\gamma'}$

Also \begin{align*} \abs{\gamma'}&\leq \abs{\beta_1}+\abs{\gamma}+\abs{\beta_2}\\&\leq (r+r_0)+\abs{\gamma}+(r+r_0)\\&\leq \abs{\gamma}+4r \end{align*} and $\gamma$ is an arbitrary path in $C_{\rho ({r}/{2})}(B)$ connecting $y_1$, $y_2$. 

Thus, \[d_{\rho r/4,A}(x_1,x_2)\leq d_{\rho ({r}/{2}),B}(y_1,y_2)+4r.\]
Therefore, \[\delta^n_{\rho/4}(r)\leq \delta'^{6n}_{\rho}\Big(\frac{r}{2}\Big)+4r.\]
Thus, $\delta^n_{\rho/4} \preceq \delta'^{6n}_{\rho}$. \qedhere
\end{proof}

We now finish the proof of Proposition~ \ref{qiv}.

\begin{proof}
Let $\Phi$ be a map from $X$ to $Y$ such that the Hausdorff distance between $\Phi(A)$ and $B$ is finite. Then $Div(X,A) \preceq Div\bigl(Y,\Phi(A)\bigr)$ by Lemma~ \ref{qiv1} and $Div\bigl(Y,\Phi(A)\bigr) \sim Div(Y,B)$ by Lemma~ \ref{qiv2}. Thus, $Div(X,A)\preceq Div(Y,B)$.
Similarly, $Div(Y,B)\preceq Div(X,A)$.
Therefore, $Div(X,A)\sim Div(Y,B)$. \qedhere
\end{proof} 

We are now ready to define the concept of relative upper divergence of of a finitely generated group with respect to a subgroup.
\begin{defn} 
Let $G$ be a finitely generated group and $H$ its subgroup. We define \emph{the relative upper divergence} of $G$ with respect to $H$, denoted \emph{$Div(G,H)$} to be the relative upper divergence of the Cayley graph $\Gamma(G,S)$ with respect to $H$ for some finite generating set $S$.
\end{defn}

\begin{rem}
\label{sudofgg}
If $H$ is the trivial subgroup, then $\delta^n_\rho=\delta^2_\rho$ for all $n\geq 2$. Thus, we can ignore the parameter $n$ in the family $\{\delta^n_{\rho}\}$ and consider that $Div(G,e)$ is characterized by the one-parametrized family of functions $\{\delta_{\rho}\}$. By this way, the upper relative divergence $Div(G,e)$ is the same as the upper divergence $Div(G)$ of the group $G$ in terms of Gersten \cite{MR1254309} 
\end{rem}
\subsection{Relative lower divergence}
\label{rld}

In this section, we introduce the concept of relative lower divergence of geodesic spaces as well as finitely generated groups. Similar to upper divergence, this concept is also a quasi-isometry invariant.
\begin{defn}
Let $(X,A)$ be a pair of spaces. For each $\rho \in (0,1]$ and positive integer $n\geq 2$, we define a function $\sigma^n_{\rho}\!:[0, \infty) \to [0, \infty]$ as follows: 

For each positive $r$, if there is no pair of $x_1, x_2 \in \partial N_r(A)$ such that $d_X(x_1,x_2)\geq nr$ and $d_r(x_1, x_2)< \infty$, we define $\sigma^n_{\rho}(r)=\infty$.

Otherwise, we define $\sigma^n_{\rho}(r)=\inf d_{\rho r}(x_1,x_2)$ where the infimum is taken over all $x_1, x_2 \in \partial N_r(A)$ such that $d_r(x_1, x_2)<\infty$ and $d(x_1,x_2)\geq nr$.
%\[\sigma^n_{\rho}(r)=\inf\set{d_{C_{\rho r}(A)}(x_1, x_2)}{\text{$x_1$, $x_2$ $\in \partial N_r(A)$, $d_{C_r(A)}(x_1, x_2)< \infty$, and $d_S(x_1,x_2)\geq nr$}}.\]

%\[\sigma^n_{\rho}(r)=\inf_{\substack{x_1,x_2\in \partial N_r(A)\\d_r(x_1, x_2)< \infty\\d_S(x_1,x_2)\geq nr}} d_{\rho r}(x_1, x_2).\]

The family of functions $\{\sigma^n_{\rho}\}$ is \emph{the relative lower divergence} of $X$ with respect $A$, denoted $div(X,A)$.
\end{defn}

By using the same argument from the previous section, we have the following proposition.

\begin{prop}
\label{qiv'}
If two pairs of spaces $(X,A)$ and $(Y,B)$ are quasi-isometric, then $div(X,A)\sim div(Y,B)$.
\end{prop}

We now define the concept of relative lower divergence of a finitely generated group with respect to a subgroup.
\begin{defn} 
Let $G$ be a finitely generated group and $H$ its subgroup. We define \emph{the relative lower divergence} of $G$ with respect to $H$, denoted \emph{$div(G,H)$}, to be the relative lower divergence of the Cayley graph $\Gamma(G,S)$ with respect to $H$ for some finite generating set $S$. 
\end{defn}

\subsection{Some Properties of Relative Divergence of finitely generated groups}
In this section, we examine some key properties of relative divergence and we compare upper and lower relative divergence.

\begin{thm}
Let $G$ be a finitely generated group and $H$ a subgroup of $G$. Suppose that $Div(G,H)=\{\delta^n_{\rho}\}$ and $div(G,H)=\{\sigma^n_{\rho}\}$. 
\begin{enumerate}
\item If $H$ is an infinite index subgroup of $G$, then $\delta^n_{\rho}(r)<\infty$ for every $r>0$.
\item If $H$ is infinite and $0<\tilde{e}(G,H)<\infty$, then $\sigma^n_{\rho}(r)<\infty$ for every $r>0$.
\end{enumerate}
\end{thm}
\begin{proof}
Fix a finite set $S$ of generators of $G$. %Fix an $H$--perpendicular rays $\gamma$ with the initial point at the identity element. 

First, we will prove that $\delta^n_{\rho}(r)<\infty$ for every $r>0$. %Let $t=\lfloor r\rfloor$ and let $g_0=\gamma(t)$ be a vertex of $\Gamma(G,S)$.
 We define 
\[A=S(e,r)\cap \partial N_r(H).\] Obviously, $A$ is a non-empty finite set. We define 
\[B=\bigset{(x,y)}{\text{$x\in A$, $y\in \partial N_r(H)$, $d_r(x,y)<\infty$ and $d_S(x,y)\leq nr$}}.\]
Therefore, $B$ is also a non-empty finite set. Define $M= \bigset{d_{\rho r}(x,y)}{(x,y)\in B}$ and we will show $\delta^n_{\rho}(r)\leq M$.

Indeed, let $x$, $y$ be arbitrary points in $\partial N_r(H)$ such that $d_r(x,y)<\infty$ and $d_S(x,y)\leq nr$. Let $h$ be an element in $H$ such that $d_S(x,H)=d_S(x,h)=r$. Thus, $(h^{-1}x,h^{-1}y)\in B$ and $d_{\rho r}(x,y)=d_{\rho r}(h^{-1}x,h^{-1}y)$. Thus, $d_{\rho r}(x,y)\leq M$. It follows that $\delta^n_{\rho}(r)\leq M$. 

We now assume that $0<\tilde{e}(G,H)<\infty$ and we will prove $\sigma^n_{\rho}(r)<\infty$ for all $r>0$. Let $m=\tilde{e}(G,H)$. For each $i\in \{0,1,2,\cdots,m\}$ we could choose $h_i$ in $H$ such that the distance between $h_i$ and $h_j$ is at least $(n+2)r$ whenever $i\ne j$. By Lemma~ \ref{pr}, for each $i\in \{0,1,2,\cdots,m\}$ we could choose an $H$--perpendicular ray $\gamma_i$ with the initial point $h_i$. Thus, there are at least two different rays $\gamma_i$ and $\gamma_j$ such that $\gamma_i\cap C_r(H)$ and $\gamma_i\cap C_r(H)$ lie in the same component of $C_r(H)$. We define $u=\gamma_i(r)$ and $v=\gamma_j(r)$. Then $u$, $v$ lie in $\partial N_r(H)$, the distance $d_r(u,v)<\infty$ and $d_S(u,v)\geq nr$. Thus, 
\[\sigma^n_{\rho}(r)\leq d_{\rho r}(x,y)<\infty.\qedhere\] 
\end{proof}

\begin{thm}
\label{udld}
Let $G$ be an infinite finitely generated group and $H$ an infinite finitely generated subgroup of $G$. If $0<\tilde{e}(G,H)<\infty$, then $div(G,H)\preceq Div(G,H)$ 
\end{thm}
\begin{proof}
Fix a finite generating set $S$ of $G$ such that $T=S\cap H$ generates $H$. We could consider $\Gamma(H,T)$ as a subgraph of $\Gamma(G,S)$. We denote $Div(G,H)=\{\delta^n_{\rho}\}$ and $div(G,H)=\{\sigma^n_{\rho}\}$. Let $m=\tilde{e}(G,H)$ and $M=4(2m+1)$. We will show $\sigma^n_\rho \preceq \delta^{Mn}_{\rho}$. More precisely, we are going to prove that for each $r>2$ \[\sigma^n_\rho(r)\leq \delta^{Mn}_{\rho}(r).\] 

For each $i\in \{0, 1, 2, \cdots, m\}$ we choose $h_i$ in $H$ such that $4nir \leq {\abs{h_i}}_S< 4nir+1$ and $\gamma_i$ to be an $H$--perpendicular geodesic ray with the initial point $h_i$. Since $m=\tilde{e}(G,H)$, then there are two different geodesics $\gamma_i$ and $\gamma_j$ ($i<j$) such that $\gamma_i\cap C_r(H)$ and $\gamma_j\cap C_r(H)$ lie in the same component of $C_r(H)$. We define $x=\gamma_i(r)$ and $y=\gamma_j(r)$, then $x$ and $y$ lie in $\partial N_r(H)$ and $d_r(x, y)< \infty$.
Also, \begin{align*}d_S(x,y)&\leq d_S(x,h_i)+d_S(h_i,h_j)+d_S(h_j,y)\\&\leq r+4n(i+j)r+2+r\leq 8mnr+4r\leq (Mn)r \end{align*}

and \begin{align*}d_S(x,y)&\geq d_S(h_i,h_j)-d_S(h_i,x)-d_S(h_j,y)\\&\geq 4njr-4nir-1-r-r\geq 4nr-3r\geq nr \end{align*}

Thus, \[\sigma^n_\rho(r)\leq d_{\rho r}(x,y)\leq \sigma^{Mn}_{\rho}(r).\]

Therefore, $\sigma^n_\rho \preceq \delta^{Mn}_{\rho}$. \qedhere
\end{proof}

\begin{thm} [Commensurability]
Let $G$ be a finitely generated group. 
\begin{enumerate}
\label{poulrd}
\item If $K\leq H \leq G$ and $[H:K]<\infty$, then $Div(G,H)\sim Div(G,K)$ and $div(G,H)\sim div(G,K)$.
\item If $H_1$ and $H_2$ are two commensurable subgroups of $G$. Then, $Div(G,H_1)\sim Div(G,H_2)$ and $div(G,H_1)\sim div(G,H_2)$.
\item If $K\leq H \leq G$ and $[G:H]<\infty$, then $Div(G,K)\sim Div(H,K)$ and $div(G,K)\sim div(H,K)$.
\item For any conjugate $gHg^{-1}$ of $H$, $Div(G,gHg^{-1})\sim Div(G,H)$ and $div(G,gHg^{-1}) \sim div(G,H)$
\end{enumerate} 
\end{thm}

\begin{proof}
The theorem follows immediately from Proposition~ \ref{qiv} and Proposition~ \ref{qiv'}. \qedhere
\end{proof}

\section{Relative divergence of finitely generated groups with respect to their normal subgroups}
\label{rdns}
In this section, we investigate the upper and lower divergence of a finitely generated group relative to a normal subgroup. 

\begin{lem}
\label{nmsg}
Let $G$ be a group with a finite generating set $S$ and $H$ a normal subgroup of $G$. Suppose $g_1H$, $g_2H$ are arbitrary left cosets of $H$ and the distance between them is $n$. Then for any element $g_1h$ in $g_1H$ the distance between $g_1h$ and $g_2H$ is also $n$.
\end{lem} 
\begin{proof}
Obviously, the distance between $g_1h$ and $g_2H$ is at least $n$. Thus, we only need to show this distance is bounded above by $n$. Choose $g_1h_1$ in $g_1H$ and $g_2h_2$ in $g_2H$ such that the distance between them is $n$. Define $g= g_1hh_1^{-1}g_1^{-1}$. Since $H$ is a normal subgroup, then $g$ lies in $H$ and $g'=g(g_2h_2)$ is an element in $g_2H$. Also, $d_S(g_1h,g')=d_S(gg_1h_1,gg_2h_2)=d_S(g_1h_1,g_2h_2)=n$. Therefore, the distance between $g_1h$ and $g_2H$ is at most $n$. \qedhere
\end{proof}

\begin{thm}
\label{udns}
Let $G$ be a finitely generated group and $H$ a finitely generated normal subgroup of $G$. Suppose that $Div(G,H)=\{\delta^n_{\rho}\}$ and $Div(G/H,e)=\{\delta_\rho\}$. Let 
$\bar{\delta^n_\rho}(r)= \delta_\rho(r)+nr$ for each positive $r$. Then $Div(G/H,e)\preceq Div(G,H)\preceq \{Dist^H_G\circ\bar{\delta^n_\rho}\}$. Moreover, if $G/H$ is one-ended, then $Div(G/H,e)\preceq Div(G,H)\preceq Dist^H_G \circ Div(G/H,e)$.
\end{thm}
\begin{proof}
Let $S$ be a finite generating set of $G$ and assume that $T=G\cap S$ generates $H$. Moreover, the image $\bar{S}$ of $S$ under the quotient map is a finite generating set of the quotient group $G/H$. We see that the Cayley graph $\Gamma(G/H,\bar{S})$ is the quotient graph of the Cayley graph $\Gamma(G,S)$ under the action of $H$.

We will first show that $\delta^n_{\rho}\preceq Dist^H_G\circ\bar{\delta^n_\rho}$. More precisely, we will show that 
$\delta^n_{\rho}(r)\leq 2Dist^H_G\circ\bar{\delta^n_\rho}(r)$ for all positive $r$.

Indeed, let $x$, $y$ be arbitrary points in $\partial N_r(H)$ such that $d_{r,H}(x,y)<\infty$ and $d_S(x,y)\leq nr$. We assume that $r$ is an integer and $x$, $y$ are vertices. Thus, there is a path in $C_r(H)$ connecting $x$ and $y$. Let $\bar{x}$ and $\bar{y}$ be the associated points of $x$ and $y$ respectively in $\Gamma(G/H,\bar{S})$. Thus, $\bar{x}$ and $\bar{y}$ lie in the sphere $S_r(\bar{e})$ and there is a path outside the ball $B_r(\bar{e})$ connecting them.

Since $d_{\rho r,\bar{e}}(\bar{x},\bar{y})\leq \delta_\rho(r)$, then there is a path $\alpha$ in $C_{\rho r}(\bar{e})$ connecting $\bar{x}$, $\bar{y}$ such that the length of $\alpha$ is bounded above by $\delta_\rho(r)$. Thus, there is a path $\beta$ in $C_{\rho r}(H)$ connecting $x$ and some point $y'$ in $\partial N_r(H)$. Moreover, $y'=hy$ for some $h$, and $\alpha$, $\beta$ have the same length. Thus, the length of $\beta$ is also bounded above by $\delta_\rho(r)$. Thus, the distance between $x$ and $y'$ is also bounded above by $\delta_\rho(r)$. Therefore, the distance between $y$ and $y'$ is bounded above by $\delta_\rho(r)+nr$. Since $y$ and $y'$ lie in the same left coset $gH$, then there is a path $\gamma$ with vertices in $gH$ connecting $y$ and $y'$. Thus, the path $\gamma$ must lie in $C_r(H)$ by Lemma~ \ref{nmsg}. Moreover, the path $\gamma$ can be chosen with the length bounded above by $Dist^H_G\bigl(\delta_\rho(r)+nr\bigr)$. We define 
$\beta'=\beta \cup \gamma$ then $\beta'$ is a path in $C_{\rho r}(H)$ connecting $x$, $y$ and the length of $\beta'$ is bounded above by 
$Dist^H_G\bigl(\delta_\rho(r)+nr\bigr)+\delta_\rho(r)$. Thus
\[d_{\rho r,H}(x,y)\leq Dist^H_G\bigl(\delta_\rho(r)+nr\bigr)+\delta_\rho(r)\leq 2Dist^H_G\circ\bar{\delta^n_\rho}(r).\]
Therefore, \[\delta^n_{\rho}(r)\leq 2Dist^H_G\circ\bar{\delta^n_\rho}(r).\]
Thus, \[\delta^n_{\rho}\preceq Dist^H_G\circ\bar{\delta^n_\rho}.\] 

We now show $\delta_{\rho} \preceq \delta^n_{\rho}$. More precisely, we are going to show that 
$\delta_{\rho}(r) \leq \delta^n_{\rho}(r)$ for all positive $r$.

Indeed, let $u$ and $v$ be arbitrary points in $S_r(\bar{e})$ of $\Gamma(G/H,\bar{S})$ and $d_{r,\bar{e}}(u,v)<\infty$. We assume that $r$ is an integer and $u$, $v$ are vertices. Choose 
$x_1$ and $y_1$ be lifting points of $u$ and $v$ respectively such that $d_S(x_1,y_1)=d_{\bar{S}}(u,v)\leq 2r \leq nr$. Obviously, $x_1$ and $y_1$ lie in $\partial N_r(H)$. We will show $d_{r,H}(x_1,y_1)<\infty$.

Indeed, since there is a path in $C_r(\bar{e})$ connecting $u$ and $v$, then there is a path $\alpha_1$ in $C_r(H)$ connecting two points $x_1$ and some point $y'_1$, where $y'_1=h'y_1$ for some $h'$ in $H$. Since $y_1$ and $y'_1$ lie in the same left coset $g'H$, then there is a path $\alpha_2$ with vertices in $g'H$ connecting $y_1$ and $y'_1$. By Lemma~ \ref{nmsg}, the path $\alpha_2$ also lies in $C_r(H)$. By concatenating $\alpha_1$ and $\alpha_2$, we have a path in $C_r(H)$ connecting $x_1$ and $y_1$. Thus, $d_{r,H}(x_1,y_1)<\infty$.

We now prove that $d_{\rho r,\bar{e}}(u,v)\leq d_{\rho r,H}(x_1,y_1)$. Indeed, for any path $\gamma'$ in $C_{\rho r}(H)$ connecting $x_1$ and $y_1$, there is a path 
$\bar{\gamma'}$ connecting $u$, $v$ such that the length of $\bar{\gamma'}$ is less than or equal to the length of $\gamma'$. Thus, $d_{\rho r,\bar{e}}(u,v)\leq d_{\rho r,H}(x_1,y_1)$. Therefore, $\delta_{\rho}(r) \leq \delta^n_{\rho}(r)$. Thus, $\delta_{\rho} \preceq \delta^n_{\rho}$.

If a quotient group $G/H$ is one-ended, then $\delta_{\rho}(r)\geq 2r$ for each $r>0$. Thus, 
\[\bar{\delta^n_\rho}(r)= \delta_\rho(r)+nr\leq (n+1)\delta_\rho(r).\]
Therefore, \[\delta^n_{\rho}(r)\leq 2Dist^H_G\circ\bar{\delta^n_\rho}(r)\leq 2Dist^H_G\circ \delta_\rho\bigl((n+1)r\bigr).\]

Thus, \[\delta^n_{\rho}\preceq Dist^H_G\circ \delta_\rho.\]

Therefore, \[Div(G,H)\preceq Dist^H_G \circ Div(G/H,e).\] \qedhere

\end{proof}

\begin{rem}
If $G=H \times K$ and $K$ is a one-ended group, then $Div(G,H)\sim Div(K,e)$. Thus, we could have any desired relative upper divergence $Div(G,H)$ by controlling the divergence $Div(K,e)$. In particular, any finitely generated group $H$ could be embedded as a subgroup of a larger finitely generated group $G$ such that $Div(G,H)$ is any polynomial functions or exponential function. Indeed, we only need to choose $K$ to be a one-ended hyperbolic group to have the upper relative divergence $Div(G,H)$ as the exponential function. Similarly, we can choose a one-ended group $K$ such that $Div(K,e)$ is equivalent to a desired polynomial (for example, see \cite{MR3032700}) and $Div(G,H)$ is also equivalent to this desired polynomial.
\end{rem}

\begin{thm}
\label{ldns}
Let $G$ be a finitely generated group and $H$ an infinite normal subgroup of $G$. Let $K$ be any finitely generated subgroup of $H$. Then, $div(G,H)\preceq dist^K_G$. In particular, if $H$ is finitely generated, then $div(G,H)\preceq dist^H_G$.
\end{thm}
\begin{proof}
Let $S$ be a finite generating set of $G$ and assume that $T=K\cap S$ generates $K$. Thus, $\Gamma(K,T)$ is a subgraph of $\Gamma(G,S)$. Denote $div(G,H)=\{\sigma^n_{\rho}\}$. We will prove that $\sigma^n_{\rho}\preceq dist^K_G$. More precisely, $\sigma^n_{\rho}(r)\leq dist^K_G(nr)$.
 
For each $r>0$, we assume that $r$ is an integer. Since $dist^K_G(nr)=\min \bigset{\abs{k}_T}{\abs{k}_S\geq nr}$, then there is an element $k_0$ in $K$ such that $\abs{k_0}_S\geq nr$ and $\abs{k_0}_T\leq dist^K_G(nr)$. Let $\alpha$ be a geodesic in $\Gamma(K,T)$ connecting the identity element $e$ and $k_0$. Thus, all vertices of $\alpha$ lie in $H$, and the length of $\alpha$ is bounded above by $dist^K_G(nr)$. Choose any element $g$ in $G$ such that $d_S(g,H)=r$ and define $x=g$ and $y=gk_0$. By Lemma~ \ref{nmsg}, the points $x$ and $y$ lie in $\partial N_r(H)$ and $g\alpha$ is a path in $C_r(H)$ connecting $x$ and $y$. Moreover, $d_S(x,y)=\abs{k_0}_S\geq nr$. Thus, \[\sigma^n_{\rho}(r)\leq d_{\rho r}(x,y)\leq \ell(g\alpha)\leq\ell(\alpha)\leq dist^K_G(nr).\] Therefore, $\sigma^n_{\rho}\preceq dist^K_G$. \qedhere
\end{proof}

\begin{cor}
\label{csifgg}
Let $G$ be a finitely generated group and $H$ an infinite normal subgroup of $G$. If $H$ contains some infinite finitely generated subgroup, then $div(G,H)$ is dominated by the growth of $G$. In particular, $div(G,H)$ is at most exponential.
\end{cor} 

\begin{rem}
In Corollary \ref{csifgg}, it is unknown whether or not $div(G,H)$ is dominated by the exponential function when every finitely generated subgroup of $H$ is finite.

In Theorem~ \ref{ldns}, the relative lower divergence $div(G,H)$ can be strictly dominated by $dist^H_G$. Similarly, $Div(G,H)$ could be strictly dominated by $Dist^H_G\circ Div(G/H,e)$ in Theorem~ \ref{udns} (here we assume that $G/H$ is one-ended). We now compute the relative divergence of the Heisenberg group with respect to some cyclic subgroup to show these facts.
\end{rem}
Before computing the relative divergence of the Heisenberg group with respect to some cyclic subgroup, we need some results about this group.

\begin{lem}
\label{hs}
 Let $G=\langle a,b,c| bab^{-1}a^{-1}=c, ac=ca, bc=cb\rangle$ be the Heisenberg group and $H$ the cyclic subgroup generated by $c$. Then 
\begin{enumerate}
\item Each element of $G$ could be written uniquely in the form $a^kb^\ell c^p$, where $k, \ell, p$ are integers.
\item \begin{align*}(a^kb^\ell c^p)a&=a^{k+1}b^\ell c^{p+l}\\(a^kb^\ell c^p)b&=a^kb^{\ell+1}c^p\\(a^kb^\ell c^p)c&=a^kb^\ell c^{p+1}\end{align*}
\item $H$ is a normal subgroup of $G$, and $G/H=\mathbb{Z}^2$ is one-ended.
\item If $\abs{a^kb^\ell c^p}\leq N$, then $\abs{k}\leq N$, $\abs{\ell}\leq N$, $\abs{p}\leq N^2$
\item $d_S(a^kb^\ell c^p,H)=\abs{k}+\abs{\ell}$
\end{enumerate}
\end{lem} 
\begin{proof}
For the facts (1), (2), (3) and (4), we refer the reader to Examples 1.5 and 1.18 in \cite{MR1086648}. We now prove the fact (5).

First we observe that $c$ commutes with every element of group $G$.
Since $d_S(a^kb^\ell c^p,c^p)=d_S(c^pa^kb^\ell,c^p)=\abs{a^kb^\ell}_S\leq \abs{k}+\abs{\ell}$ and $c^p\in H$, then $d_S(a^kb^\ell c^p,H)\leq\abs{k}+\abs{\ell}$. Let $c^{p'}$ be an element in $H$ such that 
$d_S(a^kb^\ell c^p,H)=d_S(a^kb^\ell c^p,c^{p'})$. Thus, $d_S(a^kb^\ell c^p,H)=\abs{c^{-p'}a^kb^\ell c^p}_S=\abs{a^kb^\ell c^{p-p'}}_S$. Let $w$ be the shortest word such that $a^kb^\ell c^{p-p'}\equiv_G w$. We could write $w$ in the form $w=a^{k_1}b^{\ell_1}c^{p_1}a^{k_2}b^{\ell_2}c^{p_2}\cdots a^{k_n}b^{\ell_n}c^{p_n}$ and $\abs{w}_S=\sum_{i=1}^{n} (\abs{k_i}+\abs{\ell_i}+\abs{p_i})$. We note that the values of $k_i, \ell_i, p_i$ can be zero. Thus, 
\[d_S(a^kb^\ell c^p,H)=\sum_{i=1}^{n} (\abs{k_i}+\abs{\ell_i}+\abs{p_i}).\]
Also, there is $p''$ such that $w\equiv_G a^{k_1+k_2+\cdots+k_n}b^{\ell_1+\ell_2+\cdots+\ell_n}c^{p''}$

Thus, $a^kb^\ell c^{p-p'}\equiv_G a^{k_1+k_2+\cdots+k_n}b^{\ell_1+\ell_2+\cdots+\ell_n}c^{p''}$

By (1), it implies that $k=k_1+k_2+\cdots+k_n$ and $\ell=\ell_1+\ell_2+\cdots+\ell_n$

Then, \[d_S(a^kb^\ell c^p,H)=\sum_{i=1}^{n} (\abs{k_i}+\abs{\ell_i}+\abs{p_i}) \geq \abs{k}+\abs{\ell}.\]

Therefore, $d_S(a^kb^\ell c^p,H)=\abs{k}+\abs{\ell}$. \qedhere
\end{proof}
\begin{thm}
\label{rdohs}
 Let $G=\langle a,b,c| bab^{-1}a^{-1}=c, ac=ca, bc=cb\rangle$ be the Heisenberg group and $H$ the cyclic group generated by $c$. Then 
\begin{enumerate}
\item $dist^H_G$ and $Dist^H_G$ are both quadratic.
\item $div(G,H)$ and $Div(G,H)$ are both linear.
\end{enumerate}
\end{thm}
\begin{proof}
The fact that $dist^H_G$ and $Dist^H_G$ are both quadratic could be seen in Theorem~ \ref{rd}. We see that $\tilde{e}(G,H)=e(G/H)=1$ by Theorem~ \ref{feons}. Thus, $div(G,H)\preceq Div(G,H)$ by Theorem~ \ref{udld}. Therefore, it is sufficient to show $Div(G,H)$ is linear.

Denote $Div(G,H)=\{\delta^n_{\rho}\}$. We will show that $\delta^n_{\rho}\preceq r$. More precisely, we are going to show that $\delta^n_{\rho}(r)\leq 50nr$ for all positive $r$. 

Indeed, let $x$ and $y$ be arbitrary points in $\partial N_r(H)$ such that $d_r(x,y)<\infty$ and $d_S(x,y)\leq nr$. Assume that $r$ is an integer and $x$, $y$ are vertices. Write $x=a^kb^\ell c^p$ and $y=a^{k'}b^{\ell'}c^{p'}$. Thus, $\abs{k}+\abs{\ell}=r$ and $\abs{k'}+\abs{\ell'}=r$ by Lemma~ \ref{hs}(5).

By Lemma~ \ref{hs}(2) and the fact that $c$ commutes with any element of group $G$, we compute 
\[x^{-1}y=a^{k'-k}b^{\ell'-\ell}c^{(p'-p)-\ell(k'-k)}.\]
Also, \[\abs{x^{-1}y}_S=d_S(x,y)\leq nr.\]
Thus, $\abs{k'-k}\leq nr$, $\abs{\ell'-\ell}\leq nr$ and $\abs{(p'-p)-\ell(k'-k)}\leq n^2r^2$.

Therefore, \[\abs{p'-p}\leq \abs{(p'-p)-\ell(k'-k)}+\abs{\ell(k'-k)}\leq n^2r^2+nr^2\leq 2n^2r^2.\]

Let $\ell_1$ be a number such that $\ell \ell_1\geq 0$ and $\abs{\ell_1}=r$. Let $x_1=xb^{\ell_1-\ell}$; $x_2=x_1a^{r-k}$ and $x_3=x_2b^{13nr-\ell_1}$. By Lemma~ \ref{hs}(2), we see that $x_3=a^rb^{13nr}c^{p+\ell_1(r-k)}$. 

Since $x_1=xb^{\ell_1-\ell}$ and $\abs{\ell_1-\ell}\leq r$; then there is a path $\alpha_1$ with edges labeled by $b$ connecting $x$ and $x_1$ such that the length of $\alpha_1$ is less than or equal to $r$. Similarly, there is a path $\alpha_2$ with edges labeled by $a$ connecting $x_1$, $x_2$ such that the length of $\alpha_2$ is less than $2r$ and a path $\alpha_3$ with edges labeled by $b$ connecting $x_2$, $x_3$ such that the length of $\alpha_3$ is less than $14nr$. Let $\alpha=\alpha_1\cup\alpha_2\cup\alpha_3$. We see that each vertex of $\alpha$ is of the form $x=a^{k_1}b^{\ell_1} c^{p_1}$ where $\abs{k_1}+\abs{\ell_1}\geq r$. Therefore, $\alpha$ is a path in $C_r(H)$ by Lemma~ \ref{hs}(5) and $\alpha$ connects $x$ and $x_3$, where $x_3=a^rb^{13nr}c^{p+\ell_1(r-k)}$ and $\abs{\ell_1}=r$. Moreover, the length of $\alpha$ is bounded above by $17nr$. 

By a similar argument, there is a path $\beta$ in $C_r(H)$ connecting $y$ and $y_3$, where $y_3=a^rb^{13nr}c^{p'+\ell'_1(r-k')}$ and $\abs{\ell'_1}=r$. Moreover, the length of $\beta$ is bounded above by $17nr$.

We now try to connect $x_3$ and $y_3$ by a path $\gamma$ in $C_r(H)$ with length bounded above by $14nr$. Indeed, let $p_1= p+\ell_1(r-k)$ and $p'_1= p'+\ell'_1(r-k')$ and assume that $p_1\leq p'_1$. Thus, \[\abs{p_1'-p_1}\leq\abs{p'-p}+ \abs{\ell_1(r-k)}+\abs{\ell'_1(r-k')}\leq 2n^2r^2+2r^2+2r^2\leq 4n^2r^2.\] 
Thus, $0\leq p'_1-p_1\leq 4n^2r^2$.

Let $t$ be a positive number such that $t^2\leq (p'_1-p_1)<(t+1)^2$ and let $t_1=(p'_1-p_1)-t^2$. Then $t\leq 2nr$ and $t_1\leq (t+1)^2-t^2\leq 2t+1\leq 5nr$. Also, $c^{p'_1-p_1}=c^{t^2}c^{t_1}=b^ta^tb^{-t}a^{-t}c^{t_1}$ and $y_3=x_3c^{p'_1-p_1}$. Thus, we could connect $x_3$, $y_3$ by a path $\gamma$ such that the length of $\gamma$ is bounded above by $4t+t_1$. Therefore, this length is bounded above by $13nr$. Also, the distance between $x_3$ and $H$ is $(13n+1)r$. Thus, $\gamma$ must lie in $C_r(H)$. Let $\bar{\gamma}=\alpha\cup\gamma\cup\beta$ then $\bar{\gamma}$ is a path in $C_r(H)$ connecting $x$, $y$ and the length of $\bar{\gamma}$ is bounded above by $50nr$. Thus, $d_{\rho r}(x,y)<50nr$. Therefore, $\delta^n_{\rho}(r)\leq 50nr$. Thus, $\delta^n_{\rho}\preceq r$. \qedhere 
\end{proof}

\section{Relative divergence of finitely generated groups with respect to their cyclic subgroups}
\label{rdcs}
In this section, we investigate the upper and lower divergence of a finitely generated group relative to an infinite cyclic subgroup. 

\begin{defn}
\label{axisfcg}
Let $G$ be a group with finite generating set $S$ and $H$ an infinite cyclic subgroup of $G$ generated by some element $h$ in $S$. Let $e_h$ be the edge with the identity vertex as the initial point and labeled by $h$ in $\Gamma(G,S)$. A bi-infinite arc $\alpha=\cup_{n\in \mathbb{Z}}h^ne_h$ is \emph{the axis} of $H$.
\end{defn}

Suppose $G$ is a finitely generated one-ended group and $H$ is an infinite cyclic subgroup of $G$ in this section. Let $h$ be a generator of $H$ and assume that the finite generating set $S$ of $G$ contains $h$. Let $\alpha$ be the axis of $H$. Thus, $\alpha$ is a bi-infinite arc with all vertices in $H$.

We now define the concept of divergence of a bi-infinite arc in a one-ended geodesic space. This concept will play an important role for investigating the lower divergence of a one-ended group $G$ with respect to an infinite cyclic subgroup. 
 
\begin{defn}
Let $X$ be a one-ended geodesic space and $\beta$ a proper bi-infinite arc. Let $c$ be one point on $\beta$. \emph{The divergence of $(\beta,c)$}, denoted $div(\beta,c)$, is the function $f\!:(0,\infty)\to (0,\infty)$ as follows:

For each positive $r$, we define 
\[f(r)=\inf\bigset{\abs{\gamma}}{\text{$\gamma$ is a path in $X-B(c,r)$ with endpoints on $\beta$ and on different sides of $c$}}.\]
\end{defn}
\begin{rem}
\label{rm1}
Observe that $div(\beta,c)$ is a non-decreasing function.

Let $\alpha$ be the axis of the infinite cyclic subgroup $H$, which is defined in Definition \ref{axisfcg}. Then $div(\alpha,h^i)=div(\alpha,e)$ in the Cayley graph $\Gamma(G,S)$ for any element $h^i$ in $H$ and let $div_\alpha=div(\alpha,e)$. 

For each $x$ in $\Gamma(G,S)-\alpha$ and $u$ a point in $\alpha$ such that $d_S(x,\alpha)=d_S(x,u)$, the point $u$ must be a vertex of $\Gamma(G,S)$. Thus, $N_r(\alpha)=N_r(H)$ for each $r>1$. Therefore, $\partial N_r(\alpha)=\partial N_r(H)$ and $C_r(\alpha)=C_r(H)$ for each $r>1$.
\end{rem}
\begin{defn}
Let $c$ be an arc in $\Gamma(G,S)$. If $c_0$ is any subset of $c$, the \emph{Hull} of $c_0$ in $c$, denoted $\Hull_c(c_0)$, is the smallest connected subspace of $c$ containing $c_0$.
\end{defn}

\begin{lem}
\label{doa}
Choose $r>1$ and let $n$ be a positive integer. Choose $s\geq 3 Dist^H_G\bigl((n+2)r\bigr)$. Let $a$, $b$, $c$ be three different points in $\alpha$ such that $c$ lies between $a$, $b$. Assume that $a$, $b$ lie outside the ball $B(c,s)$. Let $\gamma$ be an arc outside $B(c,s)$ connecting $a$ and $b$. Then there are two points $x$, $y$ in $\gamma\cap\partial N_r(\alpha)$ such that $d_S(x,y)\geq nr$ and the segment of $\gamma$ connecting $x$ and $y$ lies in $C_r(\alpha)$.
\end{lem}
\begin{proof}

First, we will show that $\gamma$ does not lie in the $r$--neighborhood of $\alpha$. Assume by way of contradiction that $\gamma$ lies in the $r$--neighborhood of $\alpha$. For each $G$--vertex $v$ of $\gamma$, let \[c_v=\Hull_{\alpha}\bigl(\alpha\cap \bar{B(v,r)}\bigr).\] For each edge $e$ of $\gamma$ with $G$--endpoints $v$ and $w$, let \[c_e= \Hull_{\alpha}(c_v\cup c_w).\] We see that the subsegment $[a,b]$ of $\alpha$ is covered by the sets $c_e$ for all edges $e$ of $\gamma$. In particular, $c$ lies in some $c_e$, where $e$ is an edge of $\gamma$. Therefore, $c$ lies between two vertices $u_1$ and $v_1$ of $\alpha$ whose distance from vertices of $e$ is at most $r$. Thus, the distance between $u_1$ and $v_1$ is less than $2r+1$. Therefore, the length of the subsegment $[u_1,v_1]$ of $\alpha$ is less than $Dist^H_G(2r+1)$. Thus,\[d_S(c,\gamma)\leq Dist^H_G(2r+1)+r<2 Dist^H_G\bigl((n+2)r\bigr)<s,\]
which is a contradiction. Thus, $\gamma$ does not lie in the $r$--neighborhood of $\alpha$.

Let $M=\bigset{x_i}{i\in\{0,1,2,\cdots,n\}}$ be the set of points of $\gamma$ that satisfies the following conditions:
\begin{enumerate}
\item We have $x_0=a$ and $x_n=b$.
\item For each $i\in\{1,2,\cdots,n-1\}$, the distance between $x_i$ and $\alpha$ is $r$. 
\item For each $i\in\{0,1,2,\cdots,n-1\}$, the open segment $(x_i,x_{i+1})$ does not contain any point in $\partial N_r(\alpha)$
\end{enumerate}
For each $i\in\{1,2,\cdots,n-1\}$, let $x'_i$ be a vertex of $\alpha$ such that $d_S(x_i,x'_i)=r$. We again assign $x'_0=a$ and $x'_n=b$. For each $i\in\{0,1,2,\cdots,n-1\}$, we define $d_i$ to be the subsegment of $\alpha$ that connect $x'_i$ and $x'_{i+1}$. Thus, $c$ must lie in some $d_{i_0}$. Since $(x_{i_0},x_{{i_0}+1})\cap \partial N_r(\alpha)=\emptyset$, then either $(x_{i_0},x_{{i_0}+1})\subset N_r(\alpha)$ or $(x_{i_0},x_{{i_0}+1})\cap N_r(\alpha)=\emptyset$

If $(x_{i_0},x_{{i_0}+1})\subset N_r(\alpha)$, we can use the same argument as above to show that $d_S(c,\gamma)<s$, which is a contradiction. Thus, $(x_{i_0},x_{{i_0}+1})\cap N_r(\alpha)=\emptyset$ or $(x_{i_0},x_{{i_0}+1})\subset C_r(\alpha)$.

Since the distance between $x_{i_0}$, $c$ is at least $s$ and the distance between $x'_{i_0}$, $x_{i_0}$ is $r$, then the distance between $x'_{i_0}$ and $c$ is at least $s-r$. Thus, the length of the segment of $\alpha$ connecting $x'_{i_0}$ and $c$ is at least $s-r$. Similarly, the length of the segment of $\alpha$ connecting $x'_{{i_0}+1}$ and $c$ is also at least $s-r$. Thus, the length of the segment of $\alpha$ connecting $x'_{i_0}$ and $x'_{{i_0}+1}$ is also at least $2s-2r$. Therefore, this length is strictly bounded below by $Dist^H_G\bigl((n+2)r\bigr)$. Thus, the distance in $H$ between $x'_{i_0}$ and $x'_{{i_0}+1}$ is strictly greater than $Dist^H_G\bigl((n+2)r\bigr)$. Therefore, the distance in $G$ between $x'_{i_0}$ and $x'_{{i_0}+1}$ is at least $(n+2)r$. Also, the distances $d_S(x'_{i_0}, x_{i_0})$ and $d_S(x'_{{i_0}+1},x_{{i_0}+1})$ are both $r$. Thus, the distance between $x_{i_0}$ and $x_{{i_0}+1}$ is at least $nr$. We let $x=x_{i_0}$ and $y=x_{{i_0}+1}$. \qedhere
\end{proof}

\begin{prop}
\label{ldadog}
Let $G$ be a one-ended group with a finite generating set $S$. Let $H$ be an infinite cyclic subgroup generated by some element in $S$ and $\alpha$ the axis of $H$. Then, \[div_\alpha\preceq div(G,H)\preceq div_{\alpha}\circ(3Dist^H_G)\].
\end{prop}
\begin{proof}
Denote $div(G,H)=\{\sigma^n_{\rho}\}$. 

We will first show that $\sigma^n_{\rho}\preceq div_{\alpha}\circ(3Dist^H_G)$. More precisely, we are going to show that 
$\sigma^n_{\rho}(r)\leq div_{\alpha}\circ(3Dist^H_G)\bigl((n+2)r\bigr)$ for all numbers $r>1$.

Indeed, let $s=3 Dist^H_G\bigl((n+2)r\bigr)$. Let $\gamma$ be any arc outside the ball $B(e,s)$ connecting two points $u$ and $v$ on $\alpha$ such that
 $e$ lies between $u$ and $v$. By Lemma~ \ref{doa}, there are two points $x$ and $y$ in $\gamma\cap\partial N_r(\alpha)$ such that $d_S(x,y)\geq nr$ and the segment of $\gamma$ connecting $x$ and $y$ lies in $C_r(\alpha)$. By Remark \ref{rm1}, two points $x$ and $y$ also lie in $\partial N_r(H)$. Then $d_{\rho r}(x,y)$ is bounded above by the length of $\gamma$. Therefore, $\sigma^n_{\rho}(r)$ is bounded above by the length of $\gamma$. Thus, 
\[\sigma^n_{\rho}(r)\leq div_{\alpha}(s).\]
Therefore, \[\sigma^n_{\rho}(r)\leq div_{\alpha}\circ(3Dist^H_G)\bigl((n+2)r\bigr).\] 

We now will show that $div_\alpha\preceq \sigma^n_{\rho}$ for each $n\geq20$. More precisely, we are going to show that for each $r>3$ 
\[div_\alpha(\rho r)\leq \sigma^n_{\rho}(r)+2r.\] 

Indeed, let $x_1$ and $y_1$ be arbitrary points in $\partial N_r(H)$ such that $d_X(x_1,y_1)\geq nr$ and $d_r(x_1, y_1)<\infty$. Let $\beta$ be any arc in $C_{\rho r}(H)$ connecting $x_1$ and $y_1$. Let $x_2$ and $y_2$ be vertices in $\alpha$ such that $d_S(x_1,\alpha)=d_S(x_1,x_2)=r$ and $d_S(y_1,\alpha)=d_S(y_1,y_2)=r$. Let $\beta_1$ be a geodesic connecting $x_1$ and $x_2$ and $\beta_2$ a geodesic connecting $y_1$ and $y_2$. Since the distance between $x_1$ and $y_1$ is bounded below by $nr$, then the distance between $x_2$ and $y_2$ is bounded below by $(n-2)r$. Let $h^i$ be a vertex of $\alpha$ such that $h^i$ lies between $x_2$, $y_2$ such that $x_2$, $y_2$ do not lie in the ball of center $h^i$ with radius $5r$. Let 
$\bar{\beta}=\beta_1\cup\beta\cup\beta_2$. Thus, $\bar{\beta}$ is a path outside the ball $B(h^i,\rho r)$ connecting the two points $x_2$, $y_2$ in $\alpha$ and $h^i$ lies between $x_2$, 
$y_2$. Therefore, we could have an arc $\beta'$ from $\bar{\beta}$ connecting two points $x_2$ and $y_2$. Thus, $div_\alpha(\rho r)$ is bounded above by the length of $\bar{\beta}$. Therefore, $div_\alpha(\rho r)$ is bounded above by $\abs{\beta}+2r$. Therefore, $div_\alpha(\rho r)$ is bounded above by $d_{\rho r}(x_1, y_1)+2r$. Thus, \[div_\alpha(\rho r)\leq \sigma^n_{\rho}(r)+2r.\] Therefore, \[div_\alpha\preceq \sigma^n_{\rho}.\qedhere\] 
\end{proof}

\begin{thm}
\label{ldocs}
Let $G$ be a one-ended finitely generated group and $H$ an infinite cyclic subgroup of $G$. Suppose that $div(G,H)=\{\sigma^n_{\rho}\}$ and $Div(G,e)=\{\delta_{\rho}\}$. Then 
$\sigma^n_{\rho}\preceq \delta_{\rho}\circ \bigl(({3}/{\rho})Dist^H_G\bigr)$.
\end{thm}
\begin{proof}
We will show that
$\sigma^n_{\rho}(r)\leq \delta_{\rho}\circ \bigl(({3}/{\rho})Dist^H_G\bigr)\bigl((n+2)r\bigr)$ for all number $r>1$.

Indeed, let $s=({3}/{\rho}) Dist^H_G\bigl((n+2)r\bigr)$. Choose $x$ and $y$ in $\alpha\cap S(e,s)$ such that $e$ lies between $x$ and $y$. Let $\gamma$ be an arbitrary arc outside 
$B_{\rho s}(e)$ connecting $x$ and $y$. Since $\rho s=3 Dist^H_G\bigl((n+2)r\bigr)$, then there are two points $x_1$ and $y_1$ in $\gamma\cap\partial N_r(\alpha)$ such that $d_S(x_1,y_1)\geq nr$ and the segment of $\gamma$ connecting $x_1$ and $y_1$ lies in $C_r(\alpha)$ by Lemma~ \ref{doa}. Thus, the two points $x_1$ and $y_1$ also lie in $\partial N_r(H)$ and the segment of $\gamma$ connecting $x_1$ and $y_1$ also lies in $C_r(H)$ by Remark \ref{rm1}. Thus, the distance $d_{\rho r}(x_1, y_1)$ is bounded above by the length of $\gamma$. Therefore, $\sigma^n_{\rho}(r)$ is also bounded above by the length of $\gamma$. Thus, 
\[\sigma^n_{\rho}(r)\leq \delta_{\rho}(s).\]
Therefore,
\[\sigma^n_{\rho}(r)\leq \delta_{\rho}\circ (\frac{3}{\rho}Dist^H_G)\bigl((n+2)r\bigr).\] 
Thus, $\sigma^n_{\rho}\preceq \delta_{\rho}\circ \bigl(({3}/{\rho})Dist^H_G\bigr)$. \qedhere
\end{proof}
\begin{rem}
In Theorem~ \ref{ldocs}, we could not replace $div(G,H)$ by $Div(G,H)$. For example, let $H=\mathbb{Z}$ and $K$ be any one-ended finitely generated group such that $Div(K,e)$ is super-linear. We define $G=H\times K$. Thus, $G$ is a one-ended finitely generated group and $H$ is an infinite cyclic subgroup of $G$. Then, $Dist^H_G$ is linear, $Div(G,e)$ is also linear and $Div(G,H)=Div(K,e)$ is super-linear.

Moreover, the two functions $\sigma^n_{\rho}$ and $\delta_{\rho}\circ \bigl(({3}/{\rho})Dist^H_G\bigr)$ in Theorem~ \ref{ldocs} can be equivalent in some cases (for example: $G=\mathbb{Z}^2$ and $H$ any cyclic subgroup of $G$), and $\sigma^n_{\rho}$ can be strictly dominated by $\delta_{\rho}\circ \bigl(({3}/{\rho})Dist^H_G\bigr)$ in some other cases (see Theorem~ \ref{rdohs}). 
\end{rem}

\section{Relative divergence of $\CAT(0)$ groups}
\label{rdcat(0)g}
In this section, we investigate the relative divergence of $(G,H)$ where $G$ is a $\CAT(0)$ group. We use Theorem~ \ref{udns} to build $\CAT(0)$ groups with arbitrary polynomial upper relative divergences with respect to some subgroup (see Theorem~ \ref{udocg}). We also examine the class of groups defined by Macura \cite {MR3032700} to obtain arbitrary polynomial lower relative divergence (see Corollary~ \ref{ldocg}). 

We now review some concepts and some basic properties of a $\CAT(0)$ group. We refer the reader to \cite{MR1744486} for studying more on $\CAT(0)$ groups. 

\begin{defn}
Let $X$ be a geodesic space. A \emph{geodesic triangle} $\Delta$ in $X$ consists of three points $p, q, r$ in $X$ and three geodesic segments $[p, q], [q, r], [r, p]$. A \emph{comparison triangle} for $\Delta$ in $\EE^2$ is a geodesic triangle $\bar{\Delta}$ in $\EE^2$ with vertices $\bar{p}, \bar{q}, \bar{r}$ such that $d(p, q) = d(\bar{p}, \bar{q}), d(q, r) = d(\bar{q}, \bar{r})$ and $d(r, p) = d(\bar{r}, \bar{p})$. A point $\bar{x}$ in $[\bar{q}, \bar{r}]$ is called a \emph{comparison point} for $x$ in $[q, r]$ if $d(q, x) = d(\bar{q}, \bar{x})$. Comparison points on $[p, q]$ and $[p, r]$ are defined in the same way.
\end{defn}

\begin{defn}
A geodesic triangle $\Delta$ in a geodesic space X satisfies \emph{the $\CAT(0)$ inequality} if $d(x,y)\leq d(\bar{x},\bar{y})$ for all points $x$ and $y$ on $\Delta$ and corresponding points $\bar{x},\bar{y}$ on the comparison triangle $\bar{\Delta}$ in Euclidean space $\EE^2$.
\end{defn}

\begin{defn}
\label{CAT(0)}
A geodesic space X is \emph{$\CAT(0)$ space} if every triangle in X satisfies the $\CAT(0)$ inequality.

A group is \emph{$\CAT(0)$} if it acts properly and cocompactly on some proper $\CAT(0)$ space.
\end{defn}

The proof of the following proposition can be found in \cite{MR1744486}.
\begin{prop}
Let $(X_1,d_1)$ and $(X_2,d_2)$ be $\CAT(0)$ spaces. Then the Cartesian product $X_1 \times X_2$ endowed with the metric $d$ defined by $d^2=d^2_1+d^2_2$ is a $\CAT(0)$ space.
\end{prop}

The following corollary is an immediate result of the above proposition.

\begin{cor}
\label{cpocg}
The direct product of two $\CAT(0)$ groups is a $\CAT(0)$ group.
\end{cor}

The following theorem is a direct result from Corollary III.$\Gamma$.4.8 and Theorem III.$\Gamma$.4.10 in \cite{MR1744486}.

\begin{thm}
\label{csocg}
Every finitely generated abelian subgroup of a $\CAT(0)$ group is undistorted.
\end{thm}

\begin{thm}
\label{udocg}
Let $f$ be any polynomial function or exponential function. There is a pair of groups $(G,H)$ where $G$ is a $\CAT(0)$ group and $H$ is a normal infinite cyclic subgroup of $G$ such that $Div(G,H)\sim f$.
\end{thm}
\begin{proof}
We will build the group $G$ of the form $G=K\times \mathbb{Z}$ and we choose a suitable one-ended $\CAT(0)$ groups $K$. We choose $H$ to be the $\mathbb{Z}$ factor of $G$. Thus, we observe that $Div(G,H)=Div(G/H,e)=Div(K,e)$ by Theorem~ \ref{udns}.

If $f$ is a polynomial of degree $d$, then we choose a subgroup $K$ such that $Div(K,e)$ is equivalent to $f$ (see \cite {MR3032700} for example). If $f$ is the exponential function, we choose $K$ to be a surface group of genus $g\geq 2$. Since a surface group of genus $g\geq 2$ is a $\CAT(0)$ group, then the group $G$ is also a $\CAT(0)$ group by Corollary~ \ref{cpocg}. Moreover, $K$ is a one-ended hyperbolic group, then the upper divergence of $K$ is exponential. Thus, the relative upper divergence $Div(G,H)$ is also exponential. \qedhere
\end{proof}

\begin{thm}
Let $G$ be a $\CAT(0)$ group and $H$ a normal subgroup of $G$ that contains at least one infinite order element. Then $div(G,H)$ is linear.
\end{thm}
\begin{proof}
By Theorem~ \ref{csocg}, there is an undistorted cyclic subgroup $K$ in $H$. By Theorem~ \ref{ldns}, we observe that $div(G,H)$ is linear. \qedhere
\end{proof}

We now investigate relative lower divergence of a class of $\CAT(0)$ groups introduced by Macura in \cite {MR3032700}. First, we will review this class of groups. 

For each integer $d\geq2$, we define \[G_d=\langle a_0, a_1,\cdots, a_d |a_0a_1=a_1a_0, a^{-1}_ia_0a_i=a_{i-1}, \text{for $2\leq i\leq d$}\rangle\]
and $H_d$ to be the cyclic subgroup generated by $a_d$.

Let $X_d$ be the presentation complex of $G_d$ and $\tilde{X_d}$ is the universal cover of $X_d$. The space $\tilde{X_d}$ is a $\CAT(0)$ square complex (see Macura \cite {MR3032700}). Moreover, $G_d$ is one-ended and we could consider the 1--skeleton $\tilde{X_d}^{(1)}$ of $\tilde{X_d}$ as the Cayley graph of $G_d$. Let $\alpha$ be the axis of the infinite cyclic subgroup of $H_d$ as in Definition \ref{axisfcg}. By Proposition~ \ref{ldadog} and Theorem~ \ref{csocg}, we can investigate the divergence $div_\alpha$ of $\alpha$ in $\tilde{X_d}$ to understand the lower divergence $div(G_d,H_d)$. Before computing $div_\alpha$, we need to review some results from \cite {MR3032700}.

\begin{prop}[Proposition 4.4, \cite {MR3032700}]
\label{m1}
There is a polynomial $q_d$, of degree $d$, such that for any point $O$ in $\tilde{X_d}$ and any two points $P, Q$ on the sphere $S(O,r)\subset \tilde{X_d}$, there is a path $\gamma$ in $\tilde{X_d}-B(O,r)$ connecting $P$ and $Q$ such that the length of $\gamma$ is at most $q_d(r)$
\end{prop}

\begin{prop}[Theorem 5.3, \cite {MR3032700}]
\label{m2}
There is a polynomial $p_d$, of degree $d$, such that the following holds. Let $T$ be any vertex on $\tilde{X_d}$. Let $\gamma_0$ and $\gamma_d$ be two geodesic rays issuing from $T$ such that they are the infinite concatenations of edges $a_0$ and $a_d$ respectively. For each path $\beta$ outside the ball $B(T,r)$ connecting $P\in \gamma_d$ and $Q\in \gamma_0$, the length of $\beta$ is bounded below by $p_d(r)$. 
\end{prop}

\begin{prop}
\label{dog}
The divergence $div_\alpha$ is polynomial of degree $d$.
\end{prop} 

\begin{proof}
By Proposition~ \ref{m1}, there is a polynomial $q_d$, of degree $d$ such that the following holds: Let $r$ be any positive number and $u$, $v$ two points in $S(e,r)\cap \alpha$ such that $e$ lies between $u$, $v$. There is a path outside $B(e,r)$ of length at most $q_d(r)$ connecting $u$ and $v$. Therefore, $div_\alpha$ is bounded above by $q_d$.

We now prove that $div_\alpha$ has some polynomial of degree $d$ as a lower bound. Let $p_d$ be the polynomial of degree $d$ in Proposition~ \ref{m2}. We will show 
$div_\alpha$ is bounded below by this polynomial. Indeed, for each positive $r$, let $\gamma$ be any path outside $B(e,r)$ with endpoints on $\alpha$ and on different sides of $e$ (see Figure~ \ref{ma1}). 
\begin{figure}
\begin{center}
\scalebox{.6}{\includegraphics{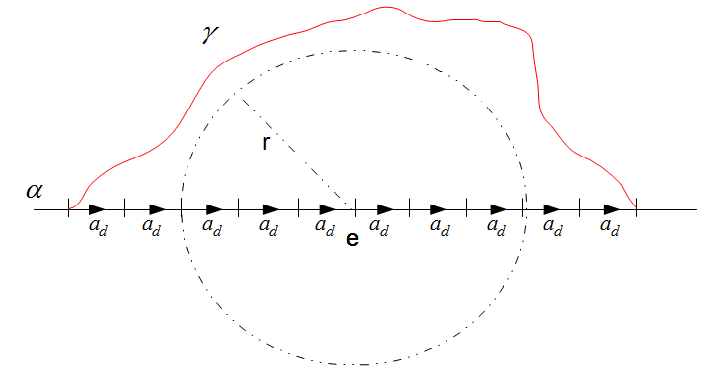}}
\end{center}
\caption{The path $\gamma$ lies outside $B(e,r)$ with endpoints on $\alpha$ and on different sides of $e$}
\label{ma1}
\end{figure}

We are going to show that there exists a subsegment $\gamma_1$ of $\gamma$ connecting two points of $\gamma_0$ and $\gamma_d$, where $\gamma_0$ and $\gamma_d$ are two geodesic rays issuing from $e$ such that they are infinite concatenations of edges $a_0$ and $a_d$ respectively (see Figure~ \ref{ma2}).

\begin{figure}
\begin{center}
\scalebox{.6}{\includegraphics{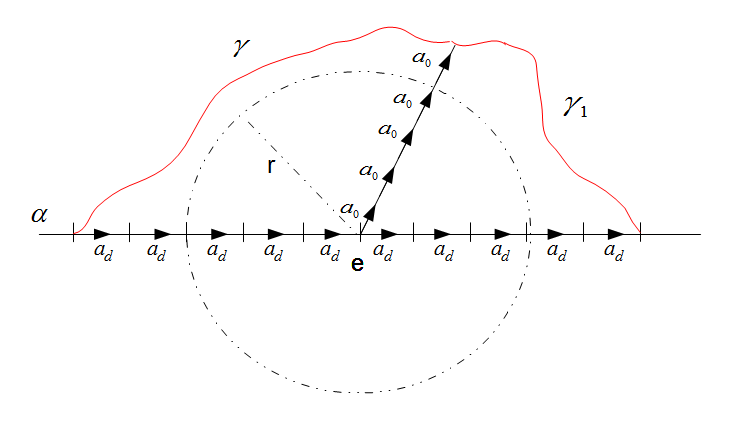}}
\caption{The subsegment $\gamma_1$ of $\gamma$ connecting two points of $\gamma_0$ and $\gamma_d$, where $\gamma_0$ and $\gamma_d$ are two geodesic rays issuing from $e$ such that they are infinite concatenations of edges $a_0$ and $a_d$ respectively}
\label{ma2}
\end{center}
\end{figure}

We will use the same technique as in \cite {MR1254309} for this argument. We observe that the path $\gamma$ and the subsegment of $\alpha$ between two endpoints of $\gamma$ form a loop in $\tilde{X_d}$ which may fill in with a reduced Van Kampen diagram $D$ (see \cite{MR0577064}). Since the path $\gamma$ lies outside the ball $B(e,r)$, the edge $a_d^{(1)}$ of $\alpha$ with the initial point $e$ must lie in some 2--cell of $D$. By the presentation of $G_d$, the edge $a_d^{(1)}$ must lie in a 2--cell $c_1$ labeled by $a_d^{-1}a_0a_da_{d-1}^{-1}$. There are two cases for $c_1$ depending on its orientation in $D$ (see Figure~ \ref{ma3}).
\begin{figure}
\centering
%\begin{center}
\begin{subfigure}{.5\textwidth}
\centering
\scalebox{0.38}{\includegraphics{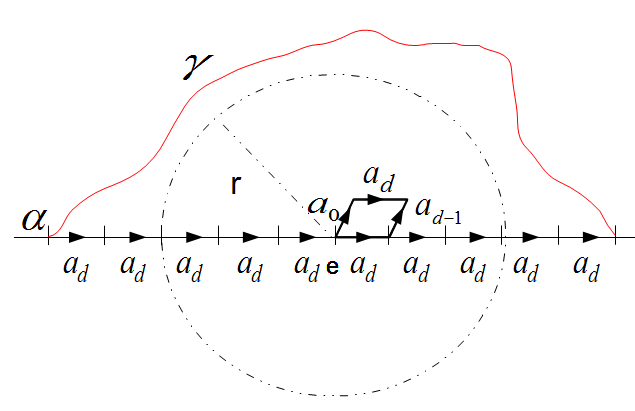}}
\caption{}
\label{ma3first}
\end{subfigure}%
\begin{subfigure}{.5\textwidth}
\centering
 \scalebox{.38}{\includegraphics{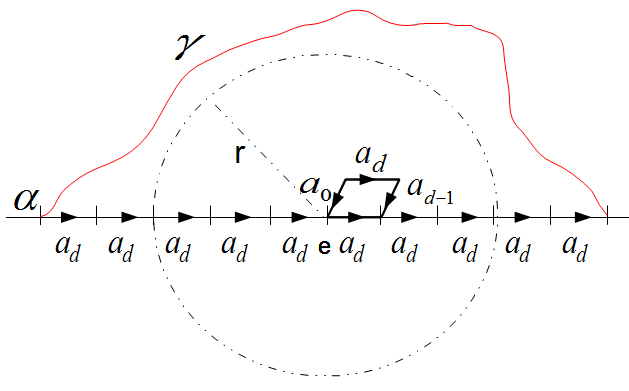}}
\caption{}
\label{ma3second}
\end{subfigure}
\caption{The position of 2--cell $c_1$ in the diagram $D$}
\label{ma3}
%\end{center}

\end{figure}

We now only argue on the first case (see Figure~ \ref{ma3first}) and the argument of the second case (see Figure~ \ref{ma3second}) is almost identical. If the edge $a_d^{(2)}$ that is opposite to $a_d^{(1)}$ in $c_1$ lies in the path $\gamma$, it is obvious that there exist a subsegment $\gamma_1$ of $\gamma$ connecting two points of $\gamma_0$ and $\gamma_d$. Otherwise, $a_d^{(2)}$ must lie in some 2--cell $c_2$ labeled by $a_d^{-1}a_0a_da_{d-1}^{-1}$ of $D$. Again, there are two possibilities for $c_2$ depending on the orientation of $c_2$ in $D$ (see Figure~ \ref{ma4}).

\begin{figure}

%\begin{center}
\begin{subfigure}{.5\textwidth}
\scalebox{.38}{\includegraphics{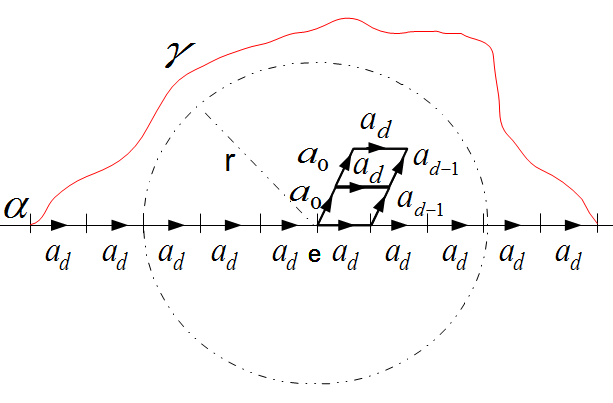}}
\caption{}
\label{ma5first}
\end{subfigure}%
\begin{subfigure}{.5\textwidth}
 \scalebox{.38}{\includegraphics{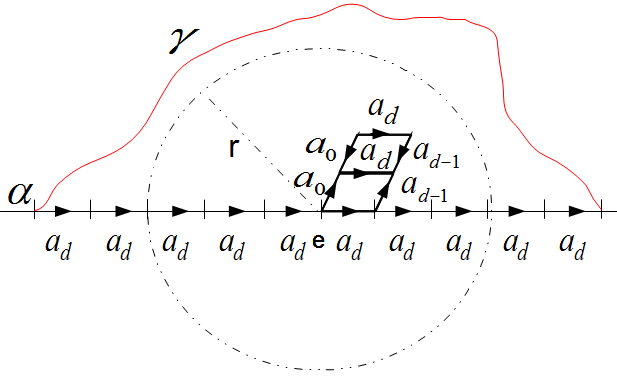}}
\caption{}
\label{ma5second}
\end{subfigure}
\caption{The position of 2--cell $c_2$ in the diagram $D$}
\label{ma4}
%\end{center}

\end{figure}

In the second case (see Figure~ \ref{ma5second}), we see that the two 2--cells $c_1$ and $c_2$ form a cancellable pair in $D$. This is impossible since the diagram $D$ is reduced. Thus, the second possibility is ruled out. By arguing inductively, we obtain a corridor that is a concatenation of 2--cells labeled by $a_d^{-1}a_0a_da_{d-1}^{-1}$ such that one edge $a_d^{(n)}$ labeled by $a_d$ of the last 2--cell in the corridor must lie in the boundary of $D$. If $a_d^{(n)}$ is an edge of $\alpha$, the diagram $D$ would not be planar topologically. Thus, $a_d^{(n)}$ must be an edge of $\gamma$ (see Figure~ \ref{ma5}). 

\begin{figure}
\begin{center}
\scalebox{.6}{\includegraphics{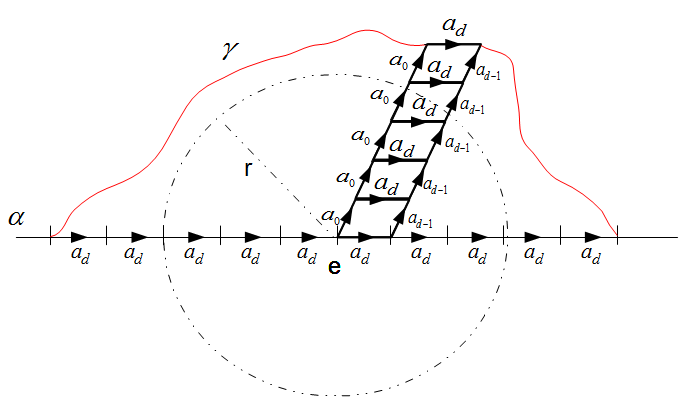}}
\end{center}
\caption{The corridor that is a concatenation of 2--cells labeled by $a_d^{-1}a_0a_da_{d-1}^{-1}$ in the diagram $D$}
\label{ma5} 
\end{figure}

Therefore, there exists a subsegment $\gamma_1$ of $\gamma$ connecting two points of $\gamma_0$ and $\gamma_d$. Since the length of $\gamma_1$ is bounded below by $p_d(r)$ by Proposition~ \ref{m2}, then the length of $\gamma$ is also bounded below by $p_d(r)$. Therefore, the divergence $div_{\alpha}$ must be dominated the polynomial $p_d(r)$. \qedhere

\end{proof}

\begin{cor}
\label{ldocg}
Let $H_d$ be a cyclic subgroup of $G_d$ generated by $a_d$. Then the relative lower divergence $div(G_d,H_d)$ is polynomial function of degree $d$.
\end{cor}
\begin{proof}
This is an immediate consequence of Proposition~ \ref{ldadog} and Proposition~ \ref{dog}. \qedhere
\end{proof}
\section{Relative divergence of relatively hyperbolic groups}
\label{rdrhg}
We now investigate the relative divergence of a relatively hyperbolic group with respect to a subgroup.

\begin{defn}
\label{defn:HyperbolicSpace}
A geodesic metric space $(X,d)$ is \emph{$\delta$--hyperbolic} if every geodesic triangle with vertices in $X$ is \emph{$\delta$--thin} in the sense that each side lies in the $\delta$--neighborhood of the union of other sides.

A finitely generated group $G$ is \emph{hyperbolic} if the Cayley graph $\Gamma(G,S)$ is a hyperbolic space for some finite set of generators $S$.
\end{defn}

\begin{defn}
\label{qc}
A subspace $Y$ of a geodesic metric space $X$ is \emph{quasi–convex} when there exists some $k > 0$ such that every geodesic in $X$ that connects a pair of points in $Y$ lies inside the 
$k$–-neighborhood of $Y$.

Suppose $G$ is a hyperbolic group with a finite generating set $S$. A subgroup $H$ of a group $G$ is quasiconvex if it is quasi-convex in the Cayley graph $\Gamma(G,S)$. 
\end{defn}

\begin{rem}
The concepts of hyperbolic groups and quasiconvex subgroups do not depend on the choice of finite set of generators (see \cite{MR1086648} and \cite{MR1170363}).
\end{rem}

We now discuss a generalization of the concepts of hyperbolic groups and quasiconvex subgroups. They are relatively hyperbolic groups and relatively quasiconvex subgroups.

\begin{defn}
Given a finitely generated group $G$ with Cayley graph $\Gamma(G,S)$ equipped with the path metric and a collection $\PP$ of subgroups of G, one can construct the \emph{coned off Cayley graph} $\hat{\Gamma}(G,S,\PP)$ as follows: For each left coset $gP$ where $P\in \PP$, add a vertex $v_{gP}$, called a \emph{peripheral vertex}, to the Cayley graph $\Gamma(G,S)$ and for each element $x$ of $gP$, add an edge $e(x,gP)$ of length 1/2 from $x$ to the vertex $v_{gP}$. This results in a metric space that may not be proper (i.e. closed balls need not be compact).
\end{defn}

\begin{defn} [Relatively hyperbolic group]
\label{rel}
A finitely generated group $G$ is \emph{hyperbolic relative to a collection $\PP$ of subgroups of $G$} if the coned off Cayley graph is $\delta$--hyperbolic and \emph{fine} (i.e. for each positive number $n$, each edge of the coned off Cayley graph is contained in only finitely many circuits of length $n$).

Each group $P\in \PP$ is a \emph{peripheral subgroup} and its left cosets are \emph{peripheral left cosets} and we denote the collection of all peripheral left cosets by $\Pi$.

An element $g$ of $G$ is \emph{hyperbolic} if $g$ is not conjugate to any element of any peripheral subgroups. 
\end{defn}

\begin{lem} [\cite{Osin06}]
\label{nops}
If $G$ is a finitely generated group which is hyperbolic relative to a collection $\PP$ of subgroups of $G$, then $\PP$ is finite. 
\end{lem}

\begin{lem}[Proposition 9.4, \cite{Hruska10}]
\label{CloseCosets}
Let $G$ be a group with a finite generating set $S$. Suppose $xH$ and $yK$
are arbitrary left cosets of subgroups of $G$.
For each constant $L$ there is a constant $L'=L'(G,S,xH,yK)$
so that in the metric space $(G,d_S)$ we have
\[
 N_L(x H)\cap N_L(y K) \subset
 N_{L'}(x H x^{-1} \cap y K y^{-1}).
\]
\end{lem}

\begin{defn}
\label{rqc}
Let $(G,\PP)$ be a relatively hyperbolic group. A subgroup $H$ of $G$ is \emph{relatively quasiconvex} if the following holds.
Let $S$ be some (any) finite generating set for $G$. Then there is a constant $\kappa=\kappa(S)$ such that
for each geodesic $\bar{c}$ in $\hat{\Gamma}(G,S,\PP)$ connecting two points of $H$,
every $G$--vertex of $\bar{c}$ lies within a $d_S$--distance $\kappa$ of $H$.
\end{defn}

\begin{rem}
We note that the concepts of relative hyperbolicity and relative quasiconvexness subgroups do not depend on the choice of finite set of generators (see \cite{Osin06}).

Throughout this section, we denote the metric in $\Gamma(G,S)$ by $d_S$ and the metric in $\hat{\Gamma}(G,S,\PP)$ by $d$.
\end{rem}
\begin{defn}
Let $(G,\PP)$ be a relatively hyperbolic group. 
\begin{enumerate}
\item A relatively quasiconvex subgroup $H$ of $G$ is \emph{strongly relatively quasiconvex} if for each conjugate $g^{-1}Pg$ of any peripheral subgroup $P$ and $H\cap g^{-1}Pg$ is a finite subgroup of $g^{-1}Pg$.
\item A relatively quasiconvex subgroup $H$ of $G$ is \emph{fully relatively quasiconvex} if for each conjugate $g^{-1}Pg$ of any peripheral subgroup $P$, $H\cap g^{-1}Pg$ is a finite subgroup or finite index subgroup of $g^{-1}Pg$.
\end{enumerate}
\end{defn}

\begin{lem}[Theorem 4.13, \cite{Osin06}]
Let $(G,\PP)$ be a relatively hyperbolic group. Let $H$ be a subgroup of $G$. Then the following conditions are equivalent:
\begin{enumerate}
\item $H$ is strongly relatively quasiconvex.
\item $H$ is generated by a finite set $T$ such that the natural map $(H,d_T) \to \hat{\Gamma}(G,S,\PP)$ is a quasi-isometric embedding.
\end{enumerate}
\end{lem}

\begin{lem} [Theorem 1.14, \cite{Osin06}]
\label{pohe}
Let $(G,\PP)$ be relatively hyperbolic groups with a finite generating set $S$. Then for any hyperbolic element $h\in G$ of infinite order, there exist $\lambda > 0$ and $c\geq 0$ such that $d(e,h^n)> \lambda \abs{n}-c$. In particular, the cyclic subgroup $H$ generated by $h$ is undistorted with respect to $(G,d_S)$ and strongly relatively quasiconvex.
\end{lem}

The following lemma is an immediate result of Proposition 2.36 in \cite{Osin06}.
\begin{lem}
\label{Osips} 
Let $(G,\PP)$ be a relatively hyperbolic groups. Then the following conditions hold:
\begin{enumerate}
\item $g_1P_1g_1^{-1}\cap g_2P_2g_2^{-1}$ is finite, where $P_1$ and $P_2$ are two different peripheral subgroups.
\item $gPg^{-1}\cap P$ is finite, where $P$ is a peripheral subgroup and $g\notin P$.
\end{enumerate}
\end{lem}

\begin{thm}[Section 8.2, \cite{MR919829}]
\label{TitsA}
Let $(G,\PP)$ be a relatively hyperbolic group and $H$ an infinite subgroup of $G$. If $H$ is not conjugate to a subgroup of any peripheral subgroup, $H$ contains a hyperbolic element. 
\end{thm}

\begin{lem}
\label{he2} 
Let $(G,\PP)$ be a relatively hyperbolic group and $H$ an infinite index, infinite normal subgroup of $G$. Then $H$ contains at least one infinite order hyperbolic element.
\end{lem}
\begin{proof}
If $H$ is not conjugate to a subgroup of any peripheral subgroup, $H$ contains a hyperbolic element by Theorem~ \ref{TitsA}. Suppose that $H$ is a subgroup of some conjugate $gPg^{-1}$ of some peripheral subgroup $P$, then $H=g^{-1}Hg$ is a subgroup of $P$. Let $g_1$ be an element in $G-P$, then $H=g_1^{-1}Hg_1$ is also a subgroup of $g_1^{-1}Pg_1$. Then, $\abs{P\cap g_1^{-1}Pg_1}=\infty$, which is contradicts Lemma~ \ref{Osips}. Therefore, $H$ is not a subgroup of any conjugate of any peripheral subgroup. \qedhere
\end{proof}

\begin{lem} [Theorem 3.26, \cite{Osin06}]
\label{t} There is a positive constant $\sigma$ such that the following holds. Let $\Delta=pqr$ be a triangle whose sides $p,q,r$ are geodesic in $\hat{\Gamma}(G,S,\PP)$. Then for each $G$--vertex $v$ on $p$, there is a $G$--vertex $u$ in the union $q\cup r$ such that $d_S(u,v)\leq \sigma$.
\end{lem}

The following lemma is an immediate result of Lemma~ \ref{t}.

\begin{lem} 
\label{q} 
There is a positive constant $\sigma$ such that the following holds. Let $pqrs$ be a quadrilateral whose sides $p,q,r,s$ are geodesic in $\hat{\Gamma}(G,S,\PP)$. Then for each $G$--vertex $v$ on $p$, there is a $G$--vertex $u$ in the union $q\cup r\cup s$ such that $d_S(u,v)\leq 2\sigma$.
\end{lem}

\begin{lem}[Lemma A.3, \cite {MR2153979}]
\label{mainlemma}
Let $(G,\PP)$ be a relatively hyperbolic group with a finite generating set $S$. Then there is a constant $K>1$ such that the following holds. Let $p$ and $q$ be paths in 
$\hat{\Gamma}(G,S,\PP)$ such that $p- =q-$, $p+ =q+$, and $q$ is geodesic in $\hat{\Gamma}(G,S,\PP)$. Then for any vertex $v \in q$, there exists a vertex $w \in p$ such that $d_S(w,v)\leq K \log_2 \abs{p}$.
\end{lem}

\begin{lem}[Lemma 4.15, \cite{MR2153979}]
\label{lemma2}
\label{thm:PeripheralQuasiconvex}
Let $(G,\PP)$ be a relatively hyperbolic group with a finite generating set $S$.
For each $A_0$ there is a constant $A_1=A_1(A_0)$ such that the following
holds in $\Cayley(G,\Set{S})$.
Let $c$ be a geodesic segment whose endpoints lie in the $A_0$--neighborhood
of a peripheral left coset $gP$.
Then $c$ lies in the $A_1$--neighborhood of $gP$.
\end{lem}

\begin{lem}[Theorem 4.1, \cite{MR2153979}]
\label{lemma4}
\label{thm:Isolated}
Suppose $(G,\PP)$ is relatively hyperbolic with a finite generating set $\Set{S}$.
For each $M, M'<\infty$ there is a constant $\iota=\iota(M,M')<\infty$
so that for any two peripheral cosets $gP \ne g'P'$ we have
\[
\diam \bigl( \nbd{gP}{M} \cap \nbd{g'P'}{M'} \bigr) < \iota
.\]
with respect to the metric $d_S$.
\end{lem}

The following concepts are introduced by Hruska (see Definition 8.9 \cite{Hruska10}) and he used it to describe the connection between geodesics in $\Gamma(G,S)$ and geodesics in $\hat{\Gamma}(G,S,\PP)$.
\begin{defn}
Let $c$ be a geodesic of $\Gamma(G,S)$,
and let $\epsilon,R$ be positive constants.
A point $x \in c$ is \emph{$(\epsilon,R)$--deep}
in a peripheral left coset $gP$ (with respect to $c$)
if $x$ is not within a distance $R$ of an endpoint of $c$ and
$\ball{x}{R} \cap c$ lies in $\nbd{gP}{\epsilon}$. A point $x \in c$ is \emph{$(\epsilon,R)$--deep} if $x$ is $(\epsilon,R)$--deep in some peripheral left coset $gP$. If $x$ is not $(\epsilon,R)$--deep in any peripheral left coset $gP$
then $x$ is an \emph{$(\epsilon,R)$--transition point} of $c$
\end{defn}

\begin{lem} [Lemma 8.10, \cite{Hruska10}] 
\label{lemma5a}
Let $(G,\PP)$
be relatively hyperbolic with a finite generating set $S$.
For each $\epsilon$ there is a constant $R=R(\epsilon)$
such that the following holds.
Let $c$ be any geodesic of $\Gamma(G,S)$,
and let $\bar{c}$ be a connected component of the set of all
$(\epsilon,R)$--deep points of $c$.
Then there is a peripheral left coset $gP$ such that each $x \in \bar{c}$
is $(\epsilon,R)$--deep in $gP$ and is not $(\epsilon,R)$--deep in any other
peripheral left coset.
\end{lem}

\begin{lem} [Proposition 8.13, \cite{Hruska10}]
\label{lemma5b}
Let $(G,\PP)$ be relatively hyperbolic with a finite generating set $S$.
There exist constants $\epsilon$, $R$ and $L$ such that the following holds.
Let $c$ be any geodesic of $\Gamma(G,S)$ with endpoints in $G$,
and let $\hat{c}$ be a geodesic of $\hat{\Gamma}(G,S,\PP)$
with the same endpoints as $c$.
Then in the metric $d_{S}$, the set of $G$--vertices of $\hat{c}$
is at a Hausdorff distance at most $L$ from the set of
$(\epsilon,R)$--transition points of $c$.
Furthermore, the constants $\epsilon$ and $R$ satisfy the conclusion
of Lemma~ \ref{lemma5a}.
\end{lem}

\begin{lem}[Lemma 4.12, \cite{MR2153979}]
\label{lemma6}
Let $(G,\PP)$ be relatively hyperbolic with a finite generating set $S$. Then for each $\theta \in [0,{1}/{2})$ there exist a number $M=M(\theta)>0$ such that for every geodesic $q$ of length $\ell$ and every peripheral left coset $gP$ with $q(0), q(\ell) \in N_{\theta \ell}(gP)$ we have $q\cap N_M(gP) \ne \emptyset$. 
\end{lem}

\begin{thm}
\label{ldonsirhg}
Let $(G,\PP)$ be a relatively hyperbolic group and $H$ an infinite index, infinite normal subgroup of $G$. Then $div(G,H)$ is linear.
\end{thm}
\begin{proof}
The proof of this theorem follows from Theorem~ \ref{ldns}, Lemma~ \ref{pohe} and Lemma~ \ref{he2}. \qedhere
\end{proof}

\begin{prop}
\label{sche}
Let $(G,\PP)$ be a relatively hyperbolic group and $H$ a subgroup of $G$ for which $H$ contains at least one infinite order hyperbolic element. If $0<\tilde{e}(G,H)<\infty$, then $Div(G,H)$ is at least exponential. 
\end{prop}
\begin{proof}
Suppose that $H$ contains an infinite order hyperbolic element $h$ and assume that $h$ is an element of the finite generating set $S$ of $G$. By Lemma~ \ref{pohe}, there is a positive integer $L$ such that $d(1,h^n)\geq ({n}/{L})-L$. Moreover, the subgroup $H_1$ generated by $h$ is strongly relatively quasiconvex. Thus, there is a constant $A>1$ such that the set of $G$--vertices of any geodesic $\beta$ in $\hat{\Gamma}(G,S,\PP)$ connecting two element of $H_1$ must lie in the $A$--neighborhood of $H_1$ with respect to the metric $d_S$. 

We define $m=\tilde{e}(G,H)$ and $M=L(12m+2L+2)$. Let $K>1$ be the constant in Lemma~ \ref{mainlemma} and let $\sigma$ the constant in Lemma~ \ref{q}. Denote $Div(G,H)=\{\delta^n_{\rho}\}$. We will prove that 
$e^r\preceq \delta^{Mn}_{\rho}$. More precisely, we define $r_0=2\sigma+({2}/{\rho})(A+2\sigma)+L+1$ and we will prove $2^{{\rho r}/{2K}}\leq \delta^{Mn}_{\rho}(r)$ for each $r>r_0$. We assume $r$ is an integer.

Indeed, for each $i\in \{0, 1, 2, \cdots, m\}$ we define $\gamma_i$ to be an $H$--perpendicular geodesic ray with the initial point $k_i=h^{L(6inr+L)}$. Since $m=\tilde{e}(G,H)$, then there are two different geodesics $\gamma_i$ and $\gamma_j$ ($i<j$) such that $\gamma_i\cap C_r(H)$ and $\gamma_j\cap C_r(H)$ lie in the same component of $C_r(H)$. We define $x=\gamma_i(r)$ and $y=\gamma_j(r)$, then $x$, $y$ lie in $\partial N_r(H)$ and $d_r(x, y)< \infty$.
Also, \begin{align*}d_S(x,y)&\leq d_S(x,k_i)+d_S(k_i,e)+d_s(e,k_j)+d_S(h_j,y)\\&\leq r+L(6inr+L)+L(6jnr+L)+r\\&\leq L(12mnr+2L)+2r\\&\leq L(12m+2L+2)nr\leq (Mn)r \end{align*}

and \begin{align*}d(k_i,k_j)&= d(e,h^{6L(j-i)nr})\\&\geq 6(j-i)nr-L\geq 12r-L\geq 6r \end{align*}
Let $\alpha_1$ be a geodesic in $\hat{\Gamma}(G,S,\PP)$ connecting $k_i$, $k_j$ and let $\alpha_2$ a geodesic in $\hat{\Gamma}(G,S,\PP)$ connecting $x$, $y$. Let $\beta_1$ be a geodesic in $\hat{\Gamma}(G,S,\PP)$ connecting $x$, $k_i$ and $\beta_2$ a geodesic in $\hat{\Gamma}(G,S,\PP)$ connecting $y$ and $k_j$. Let $u$ be a point in $\alpha_1$ such that $d(u,k_i)>2r$ and $d(u,k_j)>2r$. Thus, there is a $G$--vertex $v$ in $\beta_1\cup\alpha_2\cup\beta_2$ such that $d_S(u,v)\leq 2\sigma$.

If $v$ lies in $\beta_1$, then the distance in $\hat{\Gamma}(G,S,\PP)$ between $u$ and $k_i$ is bounded above by $r+2\sigma$. Thus, this distance is at most $2r$ which contradicts the choice of $u$. Thus, $v$ does not lie in $\beta_1$. Similarly, $v$ does not lie in $\beta_2$. Thus, $v$ must lie in $\alpha_2$. Also, $u$ lies in the $A$--neighborhood of the subgroup $H_1$ with respect to the metric $d_S$. Thus, $v$ lies in the $(A+2\sigma)$--neighborhood of $H_1$ with respect to the metric $d_S$. Therefore, the distance in $\Gamma(G,S)$ between $v$ and $H$ is bounded above by $(A+2\sigma)$.

We now prove that $d_{\rho r}(x,y)\geq 2^{{\rho r}/{2K}}$. Indeed, let $\gamma$ be any path in $C_{\rho r}(H)$ connecting $x$ and $y$. By Lemma~ \ref{mainlemma}, there exists a vertex $w \in \gamma$ such that $d_S(w,v)\leq K \log_2 \abs{\gamma}$. Since \[d_S(w,v)\geq d_S(w,H)-d_S(v,H)\geq\rho r-A-2\sigma\geq \frac{\rho r}{2},\] then 
\[K \log_2 \abs{\gamma}\geq \frac{\rho r}{2}.\]
Thus, $\abs{\gamma}\geq 2^{{\rho r}/{2K}}$. Therefore, $d_{\rho r}(x,y)\geq 2^{{\rho r}/{2K}}$. Therefore, $2^{{\rho r}/{2K}}\leq \delta^{Mn}_{\rho}(r)$. Thus, 
$e^r\preceq\delta^{Mn}_{\rho}$. \qedhere
\end{proof} 

The following is a key lemma we are going to use to investigate the lower divergence of a relatively hyperbolic group with respect to a fully relatively quasiconvex subgroup.
\begin{lem} 
\label{lemman1}
Let $(G,\PP)$ be relatively hyperbolic with a finite generating set $S$.
There exist constants $\epsilon$, $R$, $\sigma$, $K$ and $A$ such that the following hold: 
\begin{enumerate}
\item A subgroup $H$ is relatively quasiconvex if and only if there is a constant $\kappa$ such that for each geodesic $c$ in $\Gamma(G,S)$ joining points in $H$, the set of $(\epsilon, R)$--transition points of $c$ lies in the $\kappa$--neighborhood of $H$.
\item Let $\Delta=pqr$ be a triangle whose sides $p,q,r$ are geodesic in $\Gamma(G,S)$. Then for each $(\epsilon, R)$--transition point $v$ on $p$, there is an $(\epsilon, R)$--transition point $u$ in the union $q\cup r$ such that $d_S(u,v)\leq \sigma$.
\item Let $p$ and $q$ be paths in $\Gamma(G,S)$ such that $p- =q-$, $p+ =q+$ and $q$ is geodesic in $\Gamma(G,S)$. For any $(\epsilon, R)$--transition point $v \in q$, there exists a vertex $w \in p$ such that $d_S(w,v)\leq K \log_2 \abs{p}+K$.
\item For each peripheral left coset $gP$ and any geodesic $c$ with endpoints outside $N_A(gP)$. If $\ell(c)>9\max\bigl\{d_S(c^+,gP);d_S(c^-,gP)\bigr\}$, then the path $c$ contains an $(\epsilon, R)$--transition point $w$ which lies in the $A$--neighborhood of $gP$.
\end{enumerate}
 Furthermore, the constants $\epsilon$ and $R$ satisfy the conclusion of Lemma~ \ref{lemma5a}.
\end{lem}

We now give the proof for the above lemma. The reader can also find the proof of Statement (1) in \cite{Hruska10}.
\begin{proof}
Let $\epsilon$ and $R$ be constants in Lemma~ \ref{lemma5b}. Statements (1), (2), and (3) are immediate results of Definition~ \ref{rqc}, Lemma~ \ref{t}, Lemma~ \ref{mainlemma} and Lemma~ \ref{lemma5b}. We now focus on proving Statement (4).

Let 
\begin{align*}
                   &A_0=A_0\Big(\frac{1}{3}\Big) \text{ be the constant in Lemma~ \ref{lemma6}}\\
                   &A_1=A_1(A_0) \text{ be the constant in Lemma~ \ref{lemma2}}\\ 
                   &A_2=A_2(A_1,\epsilon) \text{ be the constant in Lemma~ \ref{lemma4}}\\
									 &A=A_0+A_1+A_2+\epsilon+1\\
\end{align*}	
Let $gP$ be any peripheral left coset. Let $c$ be any geodesic with endpoints outside $N_A(gP)$ such that $\ell(c)>9\max\bigl\{d_S(c^+,gP),d_S(c^-,gP)\bigr\}$. Let $r=\max\bigl\{d_S(c^+,gP),d_S(c^-,gP)\bigr\}$. Thus, the length of $c$ is greater than $9r$ and $r>A$. Since $\ell(c)>9\max\bigl\{d_S(c^+,gP),d_S(c^-,gP)\bigr\}$, then $c\cap N_{A_0}(u_1P)\ne \emptyset$ by Lemma~ \ref{lemma6}. Let $a_1$ and $a_2$ be the first points and the last points in $c\cap N_{A_0}(gP)$. Thus, the subsegment $[a_1,a_2]$ of $c$ connecting $a_1$ and $a_2$ must lie in the $A_1$--neighborhood of $gP$. Let $a'_1$ and $a'_2$ the vertices in $c-[a_1,a_2]$ such that $d_S(a_1,a'_1) \leq 1$ and $d_S(a_2,a'_2) \leq 1$. We assume that $a'_1$ lies between $c^+$, $a_1$ and that $a'_2$ lies between $c^-$, $a_2$. Obviously, $a'_1$ and $a'_2$ must lie in the $(A_0+1)$--neighborhood of $H$. In particular, they lie in the $r$--neighborhood of $H$. If the distance between $c^+$ and $a_1$ is greater than $4r$, then the distance in between $c^+$ and $a'_1$ is greater than $3r$. Thus, the subsegment of $c$ connecting $x^+$ and $a'_1$ must intersect the $A_0$--neighborhood of $gP$ which contradicts to the choice of $a_1$. Thus, $d_S(c^+,a_1)\leq 4r$. Similarly, $d_S(c^-,a_2)\leq 4r$. Also, the length of $c$ is at least $9r$. Thus, the length of $[a_1,a_2]$ is at least $r$. In particular, this length is bounded below by $A_2$.

We now show that $c$ contains an $(\epsilon,R)$--transition point $w$ in the $A$--neighborhood of $gP$. Indeed, if $[a_1,a_2]$ contains an $(\epsilon,R)$--transition point $w$, then $w$ must lie in the $A_1$--neighborhood of $gP$. In particular, $w$ lies in the $A$--neighborhood of $gP$ and we are done.

We now consider the case that $[a_1,a_2]$ contains only $(\epsilon,R)$--deep points. Therefore, $[a_1,a_2]$ lies in some $\epsilon$--neighborhood of some peripheral left coset $g'P'$. Thus, 
\[[a_1,a_2]\subset N_{A_1}(gP)\cap N_{\epsilon}(g'P').\] 
Also, the length of $[a_1,a_2]$ is at least $r$. Thus, the length of $[a_1,a_2]$ is bounded below $A_2$. Therefore, $\diam(N_{A_1}(gP)\cap N_{\epsilon}(g'P'))$ is strictly greater than $A_2$. Thus, $gP=g'P'$. It follows that $[a_1,a_2]$ lies in the $\epsilon$--neighborhood of $gP$. Also, the endpoints of $c$ both lie outside the $\epsilon$--neighborhood of $gP$. Thus, we could find an $(\epsilon,R)$--transition point $w$ in $c$ such that $d_S(w,gP)\leq \epsilon+1$. In particular, $w$ lies in the $A$--neighborhood of $gP$. \qedhere
\end{proof}

\begin{thm}
\label{ldorhg}
Let $(G,\PP)$ be a relatively hyperbolic group and $H$ an infinite fully relatively quasiconvex subgroup of $G$. If $0<\tilde{e}(G,H)<\infty$, then $div(G,H)$ is at least exponential.
\end{thm}

\begin{rem}
Before giving the proof of the theorem, we would like to discuss a large class of groups and their subgroups to which the theorem applies. More precisely, we are going to discuss different pairs of groups $(G,H)$, where $G$ is a relatively hyperbolic group and $H$ is an infinite fully relatively quasiconvex subgroup of $G$ with $0<\tilde{e}(G,H)<\infty$.

Let $G$ be the fundamental group of some hyperbolic surface and $H$ an infinite cyclic subgroup of $G$. Thus, $G$ is a hyperbolic group and $H$ is an infinite malnormal quasiconvex subgroup of $G$. In particular, $G$ is a relatively hyperbolic group and $H$ is an infinite fully relatively quasiconvex subgroup. Obviously, the number of filtered ends $\tilde{e}(G,H)=2$.

We now come up with other example. Let $G$ be the fundamental group of some hyperbolic finite volume three manifold with cusps. Therefore, $G$ is relatively hyperbolic with respect to the collection of its cusp subgroups. Let $H$ be any cusp subgroup of $G$. We can see that $H$ is an infinite fully relatively quasiconvex subgroup of $G$ and $\tilde{e}(G,H)=1$. 

We now discuss the case $H$ is a strongly relatively quasiconvex subgroup with finite number of filtered ends $\tilde{e}(G,H)$. We can choose $G$ be the fundamental group of some hyperbolic finite volume three manifold with cusps as above. Again, $G$ is relatively hyperbolic with respect to the collection of its cusp subgroups. Let $H$ be a cyclic subgroup generated by a hyperbolic element. It is obvious that $H$ is a strongly relatively quasiconvex subgroup and the number of filtered ends $\tilde{e}(G,H)=1$ 

Now, we come up with a pair of groups $(G,H)$ that satisfy all conditions in Theorem~ \ref{ldorhg} and $H$ is neither strongly relative quasiconvex nor a subgroup of some peripheral subgroup. Let $G$ be the fundamental group of some hyperbolic finite volume three manifold with more than one cusp. We can pick up any cusp subgroup $P$ and any cyclic subgroup $K$ of $G$ generated by some hyperbolic element. By Theorem 2 in \cite{MR2994828}, it is obvious that we can choose some finite index subgroup $P_1$ of $P$ and some finite index subgroup $K_1$ of $K$ such that the subgroup $H$ generated by $P_1$ and $K_1$ is isomorphic to their free product and $H$ is also a fully relatively quasiconvex subgroup. It is not hard to see that the number of filtered ends $\tilde{e}(G,H)=1$.
\end{rem}

\begin{proof}
Let $\epsilon$, $R$, $\sigma$, $K$ and $A$ be the constants in Lemma~ \ref{lemman1}

Let $\kappa$ be the constant such that for each geodesic $c$ in $\Gamma(G,S)$ joining points in $H$, the set of $(\epsilon, R)$--transition points of $c$ lies in the $\kappa$--neighborhood of $H$.

By Lemmas~ \ref{nops} and \ref{CloseCosets} we could choose $B=\max\bigset{\diam(N_\kappa(H)\cap N_\epsilon(tP)}{\text{$\abs{t}_S\leq \kappa+\epsilon$, $P\in\PP$ and $\abs{tPt^{-1}\cap H}<\infty$}}$, and we could choose $C$ such that the $C$--neighborhood of $H$ contains all peripheral left cosets $tP$ where 
$\abs{t}_S\leq\kappa+\epsilon$ and $\abs{tPt^{-1}:(tPt^{-1}\cap H)}<\infty$.

Denote $div(G,H)=\{\sigma^n_{\rho}\}$. We will prove that 
$e^r\preceq \sigma^{27n}_{\rho}$. More precisely, we define \[r_0=\frac{4C}{\rho}(\kappa+K+A+B+C+2\sigma)\] and we will prove $2^{{\rho r}/{4K}}\leq \sigma^{27n}_{\rho}(r)$ for each $r>r_0$. We assume $r$ is an integer.

Let $x$ and $y$ be arbitrary points in $\partial N_r(H)$ such that $d_S(x,y)\geq (27n)r$ and $d_r(x, y)< \infty$. (The existence of $x$ and $y$ is guaranteed by the condition $0<\tilde{e}(G,H)<\infty$.) Let $x_1$ and $y_1$ be points in $H$ such that $d_S(x,x_1)=d_S(x,H)=r$ and $d_S(y,y_1)=d_S(y,H)=r$

Let $\gamma$ be any path in $C_{\rho r}(H)$ connecting $x$ and $y$. Let $c$ be a geodesic in $\Gamma(G,S)$ connecting $x$ and $y$ and $c_1$ a geodesic in $\Gamma(G,S)$ connecting $x_1$ and $y_1$. Let $\beta_1$ be a geodesic in $\Gamma(G,S)$ connecting $x$ and $x_1$ and $\beta_2$ a geodesic in $\Gamma(G,S)$ connecting $y$ and $y_1$.

By Lemma~ \ref{lemman1}, for each $(\epsilon, R)$--transition point $u$ in $c_1$ there is an $(\epsilon, R)$--transition point $v_u$ in $\beta_1\cup c \cup \beta_2 $ such that $d_S(u,v_u)\leq 2\sigma$. We have two main cases:

Case 1: Suppose that $v_u$ lies in $c$ for some $u$ in $c_1$.

Since $u$ lies in the $\kappa$--neighborhood of $H$, then $v_u$ lies in the $(\kappa+2\sigma)$--neighborhood of $H$. By Lemma~ \ref{lemman1}, there exists a vertex $w \in \gamma$ such that $d_S(w,v_u)\leq K \log_2 \abs{\gamma}+K$. Since $w$ lies outside $N_{\rho r}(H)$, then the distance $d_S(w,v_u)$ is bounded below by $\rho r-\kappa-2\sigma$. Thus, $K\log_2 \abs{\gamma}\geq \rho r-\kappa-2\sigma-K\geq {\rho r}/{4}$ by the choice of $r$. Thus, the length of $\gamma$ is bounded below by $2^{{\rho r}/{4K}}$.	

Case 2: Suppose that $v_u$ lies in $\beta_1\cup\beta_2$ for all $(\epsilon, R)$--transition point $u$ in $c_1$.

We could choose $u_1$ and $u_2$ in $c_1$ such that $v_{u_1}\in \beta_1$, $v_{u_2}\in \beta_2$ and all points in the geodesic $c_1$ lies between $u_1$ and $u_2$ are $(\epsilon,R)$--deep points with respect to some peripheral left coset $gP$. In particular, the two points $u_1$, $u_2$ lie in the $\epsilon$--neighborhood $gP$. Since $v_{u_1}$ lies in $\beta_1$ and the length of $\beta_1$ is $r$, then the distance between $u_1$ and $x_1$ is bounded above by $r+2\sigma$. Thus, the distance between $u_1$ and $x_1$ is bounded above by $2r$ by the choice of $r$. Similarly, the distance between $u_2$ and $y_1$ is bounded above by $2r$ with respect to the metric $d_S$. By the same argument, the distances $d_S(u_1,x)$ and $d_S(u_2,y)$ are also bounded above by $2r$. Also, the distance between $x$ and $y$ is at least $(27n)r$. Thus, the distance between $u_1$ and $u_2$ is bounded below by $(27n-4)r$. Therefore, this distance is bounded below by $(23)r$ by the choice of $n$. 

Since the distance $d_S(H, gP)\leq d_S(H, u_1)+d_S(u_1,gP)\leq \kappa+\epsilon$, then there are some $h_1$ in $H$ and $t$ in $G$ such that $\abs{t}_S\leq\kappa+\epsilon$ and $gP=h_1tP$. Thus,
\begin{align*}
\diam\bigl(N_\epsilon(tP)\cap N_\kappa(H)\bigr)&=\diam\bigl(N_\epsilon(h_1tP)\cap N_\kappa(h_1H)\bigr)\\&=\diam\bigl(N_\epsilon(gP)\cap N_\kappa(H)\bigr).
\end{align*}
Since $u_1$ and $u_2$ lie in $N_\epsilon(gP)\cap N_\kappa(H)$, then \[\diam\bigl(N_\epsilon(gP)\cap N_\kappa(H)\bigr)\geq d_S(u_1,u_2)\geq (23)r>23r>r_0>B\] 
Thus, \[\diam\bigl(N_\epsilon(tP)\cap N_\kappa(H)\bigr)>B.\] Therefore, $\abs{tPt^{-1}\cap H}=\infty$ by the choice of $B$. It follows that \[\abs{tPt^{-1}:(tPt^{-1}\cap H)}<\infty\] since $H$ is a fully relatively quasiconvex subgroup. Therefore, $tP\subset N_C(H)$. Thus, \[gP=h_1tP\subset h_1N_C(H)=N_C(H).\]
Therefore, $\gamma$ lies outside the $(\rho r-C)$--neighborhood of $gP$. Thus, $\gamma$ lies outside the $({\rho r}/{2})$--neighborhood of $gP$ by the choice of $r$. 

We now show that there is an $(\epsilon,R)$--transition point $w$ in $c$ such that $d_S(w,gP)\leq A$. Since $gP$ lies in the $C$--neighborhood of $H$ and the distance between $x$ and $H$ is $r$, then $x$ lies outside the $(r-C)$--neighborhood of $gP$. In particular, $x$ lies outside the $A$--neighborhood of $gP$. Similarly, $y$ also lies outside the $A$--neighborhood of $gP$. Since the distance between $x$ and $u_1$ is bounded above by $2r$ and $u_1$ lies in the $\epsilon$--neighborhood of $gP$, then $x$ lies in the $(2r+\epsilon)$--neighborhood of $gP$. In particular, $x$ lies in the $3r$--neighborhood of $gP$. Similarly, $y$ also lies in the $3r$--neighborhood of $gP$. Since $x$ and $y$ lies in the $3r$--neighborhood of $gP$ and the distance between $x$ and $y$ is greater than $27r$, then $\ell(c)>9\max\bigl\{d_S(c^+,gP),d_S(c^-,gP)\bigr\}$, then $c$ contains an $(\epsilon,R)$--transition point $w$ in the $A$--neighborhood of $gP$ by Lemma~ \ref{lemman1}.

We now prove that the length of $\gamma$ is bounded below by $2^{{\rho r}/{4K}}$. Indeed, by Lemma~ \ref{lemman1}, there exists a vertex $v \in \gamma$ such that $d_S(v,w)\leq K \log_2 \abs{\gamma}+K$
 Also
\[d_S(v,w)\geq d_S(v,gP)-d_S(gP,w)\geq \frac{\rho r}{2}-A.\]
Thus, 
\[K\log_2 \abs{\gamma}\geq \frac{\rho r}{2}-A-K\geq \frac{\rho r}{4}.\]

Therefore, the length of $\gamma$ is bounded below by $2^{{\rho r}/{4K}}$. Thus, $d_{\rho r}(x,y)\geq 2^{{\rho r}/{4K}}$. Thus, $2^{{\rho r}/{4K}}\leq \sigma^{27n}_{\rho}$. Therefore, $e^r\preceq \sigma^{27n}_{\rho}$. \qedhere
\end{proof}

\begin{ques}
For the pair $(G,H)$ as in Theorem~ \ref{ldorhg}, is the relative lower divergence $div(G,H)$ exactly exponential? What conditions do we need to put on the pair $(G,H)$ to force the lower relative divergence $div(G,H)$ to be exactly exponential? 
\end{ques}

\begin{cor}
\label{qodc}
Let $G$ be a hyperbolic group and $H$ an infinite quasiconvex subgroup of $G$. If $0<\tilde{e}(G,H)<\infty$, then $div(G,H)$ is at least exponential.
\end{cor}

\begin{cor}
Let $(G,\PP)$ be a relatively hyperbolic group and $P$ an infinite peripheral subgroup. If $0<\tilde{e}(G,P)<\infty$, then $div(G,P)$ is at least exponential.
\end{cor}

\begin{cor}
Let $(G,\PP)$ be a relatively hyperbolic group and $H$ an infinite strongly relatively quasiconvex subgroup. If $0<\tilde{e}(G,H)<\infty$, then $div(G,H)$ is at least exponential.
\end{cor}

\begin{rem}
From the results of Corollary~ \ref{qodc} and Theorem~ \ref{ldocs}, we could extend the result of Corollary~ \ref{ldocg}. More precisely, there is a pair of groups $(G,H)$, where $G$ is a one-ended $\CAT(0)$ group and $H$ is an infinite cyclic subgroup of $G$ such that 
$div(G,H)$ is exponential. For example, let $G$ be the fundamental group of a hyperbolic surface $M$ and $H$ the fundamental group of a closed essential curve $C$ of $M$. Then $G$ is a one-ended $\CAT(0)$ group and it is also hyperbolic. Since the infinite cyclic subgroup $H$ is also quasiconvex, then $div(G,H)$ is at least exponential. Also, $div(G,H)$ is dominated by the upper divergence of $G$. Thus, $div(G,H)$ is at most exponential. Therefore, $div(G,H)$ is exactly exponential.

In Theorem~ \ref{ldorhg}, we could not replace the condition ``fully relative quasiconvexness'' of the subgroup $H$ by the condition ``relative quasiconvexness''. Readers could look at the following theorem as a counter example.
\end{rem}

\begin{thm}
\label{ldorqcs}
Let $G=\langle a_1, a_2, a_3, b,c|[a_1,a_2][a_3,b]=e, [b,c]=e\rangle$ and $H$ be the cyclic subgroup of $G$ generated by $c$. Let $P$ be the subgroup generated by $b$ and $c$. Then, 
$G$ is a relatively hyperbolic group with respect to the subgroup $P$, $0<\tilde{e}(G,H)<\infty$, $H$ is a relatively quasiconvex subgroup and $div(G,H)$ is linear.
\end{thm}

Before giving the proof of Theorem~ \ref{ldorqcs}, we need to review some result in Hruska \cite{MR2033482}

\begin{defn}[Definition 5.1, \cite{MR2033482}]
A $\CAT(0)$ 2--complex $X$ has the \emph{Isolated Flats Property} if there is a function $\Phi\!: \RR_+\to \RR_+$ such that for every pair of distinct flat planes $F_1\neq F_2$ in $X$ and for every $k \geq 0$, the intersection $N_k(F_1)\cap N_k(F_1)$ of $k$--neighborhoods of $F_1$ and $F_2$ has diameter at most $\Phi(k)$.
\end{defn}

\begin{thm}[Theorem 1.6, \cite{MR2033482}]
\label{cat(0)wisp}
Suppose a group $G$ acts properly and cocompactly by isometry on a $\CAT(0)$ 2--complex with the Isolated Flats Property. Then $G$ is hyperbolic relative to the collection of
maximal virtually abelian subgroups of rank two.
\end{thm}

We now give the proof for Theorem~ \ref{ldorqcs}.

\begin{proof}

We are going to show that $G$ acts properly and cocompactly by isometry on a $\CAT(0)$ 2--complex with the Isolated Flats Property. It is obvious that $G=G_1\underset{{<b_1>=<b_2>}}{\mathrm{*}}P$ where $G_1=\langle a_1, a_2, a_3, b_1|[a_1,a_2][a_3,b_1]=e\rangle$ and $P=\langle b_2,c|[b_2,c]=e\rangle$. Let $X_1$ be the presentation 2--complex of $G_1$ and $X_2$ the presentation 2--complex of $P$. We build the 2--complex presentation for $G$ by identifying the 1--cell $b_1$ of $X_1$ and the 1--cell $b_2$ of $X_2$ into one 1--cell called $b$. Let $\tilde{X}_1$ and $\tilde{X}_2$ be the universal covers of $X_1$ and $X_2$ respectively. It is well-known that we can put a metric on $\tilde{X}_1$ such that $\tilde{X}_1$ becomes the 2--dimensional hyperbolic plane and $G_1$ acts properly and cocompactly on $\tilde{X}_1$ by isometry. Similarly, we can put a metric on $\tilde{X}_2$ such that $\tilde{X}_2$ becomes the 2--dimensional flat and $P$ acts properly and cocompactly on $\tilde{X}_2$ by isometry. It is obvious that the universal cover $\tilde{X}$ of $X$ is the union of copies of $\tilde{X}_1$ and $\tilde{X}_2$ such that a copy of $\tilde{X}_1$ intersects a copy of $\tilde{X}_2$ in a bi-infinite arc labeled by $b$. Thus, $\tilde{X}$ is a $\CAT(0)$ 2--complex with the Isolated Flats Property. Moreover, the group $G$ acts properly and cocompactly by isometry on $\tilde{X}$. Therefore, $G$ is a relatively hyperbolic group with respect to the subgroup $P$ by Theorem~ \ref{cat(0)wisp}.

By examining the construction of $\tilde{X}$, we can see that $\tilde{e}(G,H)=1$. Moreover, $H$ is a relatively quasiconvex subgroup since it is a subgroup of peripheral subgroup $P$. We now show that the relative lower divergence $div(G,H)$ is linear.

First we show that $\abs{b^n}_S=\abs{n}$. Let $m=\abs{b^n}_S$. Obviously, $m\leq \abs{n}$. There is a homomorphism $\Phi$ from $G$ to $\mathbb{Z}$ that maps every element in $S$ to the generator of $\mathbb{Z}$. Since $m=\abs{b^n}_S$, then there is a word $w$ in $S\cup S^{-1}$ with the length $m$ such that $b^n \equiv_G w$.
Therefore,
\[ b^n\equiv_G s_1s_2\cdots s_m \text{ where $s_i \in S\cup S^{-1}$}.\] 
Thus, 
\[\Phi(b^n)=\Phi(s_1)+\Phi(s_2)+\cdots+\Phi(s_m).\]

Since $\Phi(b^n)=n$ and $\Phi(s_i)=-1 \text{ or 1}$, then $\abs{n}\leq m $. Thus, $\abs{b^n}_S=m=\abs{n}$. Similarly, $\abs{c^n}_S=\abs{n}$

We now show that $d_S(b^mc^n,H)=\abs{m}$. Denote $d_S(b^mc^n,H)=\ell$. Obviously, $\ell\leq \abs{m}$. There is a group homomorphism $\Psi$ from $G$ to $\mathbb{Z}$ that maps $b$ to the generators of $\mathbb{Z}$ and the remaining elements in $S$ to 0. Suppose that $d_S(b^mc^n,H)=d_S(b^mc^n,c^{n'})$ for some $c^{n'}$ in $H$. Thus, there is a word $w'$ with the length $\ell$ such that $b^mc^n \equiv_G c^{n'}w'$.
Therefore,
\[ b^m c^n\equiv_G c^{n'}s'_1s'_2\cdots s'_{\ell} \text{ where $s'_i \in S\cup S^{-1}$}.\] 
Thus, \[\Psi(b^m)+\Psi(c^n)=\Psi(c^{n'})+\Psi(s'_1)+\Psi(s'_2)+\cdots+\Psi(s'_\ell).\]

Since $\Psi(b^m)=m$, $\Psi(c^n)=\Psi(c^{n'})=0$ and $\Psi(s_i)=-1 \text{, 0 or 1}$, then $\abs{m} \leq \ell$. Thus, $d_S(b^mc^n,H)=\abs{m}$.

Denote $div(G,H)=\{\sigma^n_{\rho}\}$. We will prove that 
$\sigma^{n}_{\rho}$ is bounded above by a linear function. More precisely, we will show $\sigma^n_{\rho}\leq nr$ for each $r>0$. 

We assume $r$ is an integer. Let $x=b^r$ and $y=b^rc^{nr}$. Then $x$ and $y$ lie in $\partial N_r(H)$ and $d_S(x,y)\geq nr$. Let $\gamma$ be the path with vertices 
$\{b^r, b^rc,b^rc^2,\cdots,b^rc^{nr}\}$. Then, $\gamma$ is a path in $C_r(H)$ connecting $x$ and $y$. Thus, $d_r(x,y)<\infty$. Moreover, $d_{\rho r}(x,y)\leq nr$ since the length of $\gamma$ is $nr$. Thus, $\sigma^n_{\rho}\leq nr$. Therefore, $\sigma^{n}_{\rho}$ is bounded above by a linear function. \qedhere

\end{proof}

\begin{thm}
\label{urdorhp}
Let $(G,\PP)$ be a relatively hyperbolic group and $H$ a subgroup of $G$ such that $0<\tilde{e}(G,H)<\infty$. We assume that $H$ is not conjugate to any infinite index subgroup of any peripheral subgroup. %satisfies one of the following conditions:
%\begin{enumerate}
%\item $H$ is finite subgroup.
%\item $H$ contains a hyperbolic element.
%\item $H$ is conjugate to the a finite index subgroup of some peripheral subgroups. 
%\end{enumerate}
 Then $Div(G,H)$ is at least exponential.
\end{thm}

\begin{proof}
If $H$ is a finite subgroup, then the relative upper divergence $Div(G,H)$ is equivalent to the upper divergence of $G$ by Theorem~ \ref{poulrd} and Remark \ref{sudofgg}. Also, the upper divergence of $G$ is at least exponential by Sisto\cite{Sisto}. Thus, $Div(G,H)$ is at least exponential. 

In the case that $H$ is conjugate to a finite index subgroup of some peripheral subgroup. We assume that $H$ is a finite index subgroup of some peripheral subgroup by Theorem~ \ref{poulrd}. Thus, $div(G,H)$ is at least exponential by Theorem~ \ref{ldorhg}. Also, $div(G,H)$ is dominated by $Div(G,H)$ by Theorem~ \ref{udld}. Therefore, the upper relative divergence $Div(G,H)$ is also at least exponential. 

If $H$ is an infinite subgroup that is not conjugate to any subgroup of any peripheral subgroup, $H$ contains a hyperbolic element by Theorem~ \ref{TitsA}. Thus, $Div(G,H)$ is at least exponential by Proposition~ \ref{sche}. \qedhere

\end{proof} 

\begin{rem}
In Theorem~ \ref{urdorhp}, if the group $G$ is finitely presented, then the upper divergence of $G$ is exactly exponential. Therefore, the upper relative divergence $Div(G,H)$ is also exponential when the subgroup $H$ is finite. However, it is still unknown whether the upper relative divergence $Div(G,H)$ is exactly exponential in general; or what conditions we need to put on the pair $(G,H)$ to make the relative upper divergence $Div(G,H)$ to be exactly exponential.
\end{rem}

%%%%%%%%%%%%%%%%%%%%%%%%%%%%%%%%%%%%%%%%%%%%%%%%%%%%%%%%%%%
%%                BIBLIOGRAPHY
%%%%%%%%%%%%%%%%%%%%%%%%%%%%%%%%%%%%%%%%%%%%%%%%%%%%%%%%%%%
\bibliographystyle{alpha}
\bibliography{Tran}
\end{document}